\documentclass[10pt,a4paper]{article}
\usepackage[utf8]{inputenc}
\usepackage[T1]{fontenc}

\usepackage{a4wide}

\usepackage{amsmath}
\usepackage{amsthm}
\usepackage{amssymb}
\usepackage{amsopn}
\usepackage{amsfonts}
\usepackage{mathrsfs}
\usepackage{enumitem}
\usepackage{color}
\usepackage{mathtools}
\usepackage{hyperref}
\usepackage{doi}
\usepackage{graphicx}
\usepackage{bm}
\usepackage{verbatim}

\theoremstyle{plain}
\newtheorem{theorem}{Theorem}[section]

\newtheorem{lemma}[theorem]{Lemma}

\theoremstyle{definition}

\theoremstyle{remark}
\newtheorem{remark}[theorem]{Remark}

\newcommand{\rchi}{\chi}

\newcommand{\Lp}[1]{L^{#1}}

\DeclareMathOperator{\curl}{curl}
\DeclareMathOperator{\Id}{Id}
\DeclareMathOperator{\id}{id}
\DeclareMathOperator{\II}{II}
\DeclareMathOperator{\dist}{dist}
\DeclareMathOperator{\spt}{spt}  
\DeclareMathOperator{\SO}{SO}
\DeclareMathOperator{\Der}{D}
\DeclareMathOperator{\restrict}{\hspace{-0.2ex}\measurerestr\hspace{-0.2ex}}

\DeclareMathOperator{\Tan}{Tan} 
\DeclareMathOperator{\Nor}{Nor} 
\def\Xint#1{\mathchoice
   {\XXint\displaystyle\textstyle{#1}}%
   {\XXint\textstyle\scriptstyle{#1}}%
   {\XXint\scriptstyle\scriptscriptstyle{#1}}%
   {\XXint\scriptscriptstyle\scriptscriptstyle{#1}}%
   \!\int}
\def\XXint#1#2#3{{\setbox0=\hbox{$#1{#2#3}{\int}$}
     \vcenter{\hbox{$#2#3$}}\kern-.5\wd0}}

\def\dashint{\Xint-}
\newcommand{\measurerestr}{%
  \,\raisebox{-.127ex}{\reflectbox{\rotatebox[origin=br]{-90}{$\lnot$}}}\,%
}

\makeatletter
\renewcommand*\env@matrix[1][\arraystretch]{%
  \edef\arraystretch{#1}%
  \hskip -\arraycolsep
  \let\@ifnextchar\new@ifnextchar
  \array{*\c@MaxMatrixCols c}}
\makeatother

\makeatletter
\g@addto@macro\bfseries{\boldmath}
\makeatother

\newcommand{\Z}{\mathbb{Z}}

\newcommand{\R}{\mathbb{R}}

\newcommand{\eps}{\varepsilon}

\usepackage{color}

\begin{document}
\begin{center}
\begin{Large}
A Blake-Zisserman-Kirchhoff theory \\ 
for plates with soft inclusions 
\end{Large}
\\[0.5cm]
\begin{large}
Mario Santilli\footnote{Università degli studi dell'Aquila, Italy, {\tt mario.santilli@univaq.it}} and Bernd Schmidt\footnote{Universität Augsburg, Germany, {\tt bernd.schmidt@math.uni-augsburg.de}}
\end{large}
\\[0.5cm]
\today
\\[1cm]
\end{center}

\begin{abstract}
We consider a two phase elastic thin film with soft inclusions subject to bending dominated deformations. The  soft (void) phase may comprise asymptotically small droplets  within the elastic matrix. We perform a dimension reduction analysis and obtain a novel `Blake-Zisserman-Kirchhoff' functional on a natural space of `flat and fractured' two-dimensional isometric immersions that combines Kirchhoff's classical plate theory with Blake-Zisserman type surface energy contributions at cracks, folds and the boundary of voids. 
\end{abstract}

\begin{small}

\medskip

\end{small}

\tableofcontents

\section{Introduction}

We examine materials consisting of two phases: an elastic matrix and inclusions that do not show any resistance to deformations such as voids. Such materials occur in a variety of quite different applications ranging from biology to geophysics to material science to medicine. We refer to \cite{SantilliSchmidt:21} for a more detailed account of applications including references to the literature. Our focus in this contribution lies on thin structures which are of particular interest in mechanical applications. 

Mathematically such systems are modeled in the bulk by energy functionals of the form 
\begin{align}\label{eq:bulk-En}
  (y, A) 
  \mapsto \int_{\Omega \setminus A} 
    W (\nabla y(x)) \, \mathrm{d}x 
    + \int_{\Omega \cap \partial^* A} \psi (\nu(A)) \, \mathrm{d} \mathcal{H}^{2}. 
\end{align}
Here the Lipschitz domain $\Omega \subset \R^3$ is the reference configuration of the body, $y \in W^{1,2}(\Omega, \R^3)$ (say) is a deformation mapping whose elastic energy is given in terms of a stored energy function $W : \R^{3 \times 3} \to \R$. 
$A$ is a set of finite perimeter contained in $\Omega$ that represents the void part with $\partial^\ast A$ and $\nu(A)$ respectively denoting the measure-theoretic boundary and (exterior) unit-normal of $A$. We refer to  Section~\ref{sec:NotationPreliminaries} for details on the notation used in this introduction. Finally, $\psi$ is a norm on $\R^3$ that measures the surface energy per unit area of the material/void interface. 
Such energy functionals and their relaxation have been considered in \cite{BraidesChambolleSolci:07}, motivated by investigations on epitaxially strained films, cp.\ \cite{BonnetierChambolle:02, ChambolleSolci:07}; see also \cite{CrismaleFriedrich:20} for recent results and a comprehensive account of the literature in this direction. 
As it turns out, the relaxed functional acts on pairs $(y, A)$ of $SBV$ functions and sets of finite perimeter $A$ (cf.\ Section~\ref{sec:NotationPreliminaries}) where jumps in $y$ may result from thin channels forming in the void region. 

Effective theories have been derived from \eqref{eq:bulk-En} in the regime of linearized elasticity in \cite{FriedrichKreutzZemas:21} and for membranes in our recent contribution \cite{SantilliSchmidt:21}. There we have considered such functionals on thin films with $\Omega_h = \omega \times (0,h)$, $0 < h \ll 1$, and have derived an effective dimensionally reduced two phase theory for elastic membranes with soft inclusion in the limit $h \to 0$, thus extending the classical results \cite{LeDretRaoult:95,BraidesFonseca:01} for purely elastic, respectively, brittle materials. Yet, while membrane theories appropriately describe elastic deformations of rubber-like materials, they do not capture the leading order effects in stiffer (and more brittle) materials that respond elastically only to very small strains. 
In the absence of a significant plastic regime these typically develop cracks already as a result of small displacements. 
E.g., a membrane theory, even if allowing for fracture, cannot adequately describe a plate that develops bending cracks.

The envisioned extension to plates, which is the main aim of the present contribution, turns out to be considerably more involved. 
Our main goal is to describe bending dominated deformations in the elastic zone, so that the bulk energy will be of the order $h^3$, $0 < h \ll 1$ being the small film height. 
At the same time, the surface area of significant cracks will typically scale with $h$ and, hence, a naive limit $h \to 0$ for any fixed material will result in a purely elastic Kirchhoff plate theory, cp.\ \cite{FrieseckeJamesMueller:02}. 
(In fact this is inline with the experience that even very brittle materials such as glass may undergo large bending if only thin enough.) 
Yet, nontrivial results are obtained when one considers the strength of the material, which specifies the relation of the surface energy to the elastic moduli, as a second small parameter and consequently allows it to explicitly depend on $h$. 
So in effect one is lead to investigate a sequence of thin films of specific materials asymptotically as $h \to 0$. 
The most interesting scenario (and mathematically the most demanding) is when both bulk and surface energies contribute at the same scale, which in the present case leads us to considering surface energy contributions scaling with $h^2$ per unit surface area. 
Then indeed, the bending energy of the matrix and the interface energy are of the same order $h^3$, and the limiting functional will depend in a non-trivial way of both bulk and surface contributions. 
(Other scaling regimes are significantly more elementary and will lead to either a pure Kirchhoff plate theory or a degenerate trivial limiting functional.) 
We furthermore introduce a technical modeling assumption by restricting our attention to inclusions that satisfy a suitable `minimal droplet assumption', see the discussion below. 

In the absence of voids, the derivation of a dimensionally reduced theory for thin plates is a classical problem in elasticity theory, \cite{Euler,Kirchhoff,vonKarman}, also cp.\ \cite{Love,Ciarlet-II:97,Ciarlet-III:00}. 
Yet, first rigorous results on variational convergence to effective limit models are comparably recent, \cite{AcerbiButtazzoPercivale:88,AnzellottiBaldoPercivale:94,AcerbiButtazzoPercivale:91,LeDretRaoult:95}. A fundamental step towards much of the subsequent progress was achieved in \cite{FrieseckeJamesMueller:02} where a novel geometric rigidity estimate was established that carries the Korn inequality to a nonlinear setting. It allowed the authors to perform a rigorous passage from 3d nonlinear elasticity to Kirchhoff's plate theory in the bending dominated regime. As we will see, this indeed captures the behavior of the elastic matrix also in our case. Both \cite{FrieseckeJamesMueller:02} and the hierarchy of plate models found in \cite{FrieseckeJamesMueller:06} have laid the foundation to an abundance of extensions in different directions, among them shell theories \cite{FrieseckeJamesMoraMueller:03,LewickaMoraPakzad:08}, atomistic parent models \cite{Schmidt:06,BraunSchmidt:19}, non-trivial elastic response as for incompressible materials  \cite{ContiDolzmann:09} or composites with highly oscillatory elastic moduli \cite{NeukammVelvic13,HornungNeukammVelvic14,HornungPawelczykVelvic18} and, notably, multilayers and non-Euclidean plates \cite{Schmidt:07a,Schmidt:07b,LewickaMahadevanPakzad:11,BhattacharyaLewickaSchaeffner:16,LewickaLucic:18,MaorShachar:19,deBenitoSchmidt:19a,deBenitoSchmidt:19b,BoehnleinNeukammPadilla-GarzaSander:22}. Also the convergence of equilibria and dynamic solutions have been established in special cases 
\cite{MuellerPakzad:08,MoraMuellerSchultz:07,AbelsMoraMueller:11}. Of particular relevance to our set-up are extensions to brittle materials. For membranes a dimension reduction has been carried out in \cite{BraidesFonseca:01} in the static case and in \cite{Babadjian:06} for quasistatic evolutions. Beyond the membrane energy regime little appears to be known except for a recent contribution on sheets folded along a pre-assigned curve \cite{BartelsBonitoHornung:21} and the derivation of a general `Griffith-Euler-Bernoulli theory' for thin brittle beams from a nonlinear 2d Griffith functional achieved in \cite{Schmidt:17}.

In view of the above discussion on the scaling of energies we consider a Lipschitz domain $\omega \subset \mathbb{R}^2$, we set $\Omega_h = \omega \times (0,h)$ and we study energy functionals of the form
\begin{align*}
  (y, A) 
  \mapsto h^{-3} \int_{\Omega_h \setminus A} 
    W (\nabla y(x)) \, \mathrm{d}x 
    + h^{-1} \int_{\Omega_h \cap \partial^* A} \psi (\nu(A)) \, \mathrm{d} \mathcal{H}^{2}, 
\end{align*}
where $W$ has a single non-degenerate potential well at $\SO(3)$. We seek to prove a $\Gamma$-convergence type passage towards a limiting plate functional as $h \to 0$ under the hypothesis of a minimal droplet assumption for the inclusion $ A $. It guarantees that $A$ cannot contain too many small droplets of size $\ll h$ which contribute significantly to its surface measure, while we remark that droplets comparable to or larger than the film height are not restricted. We quantify this in \eqref{eq:min-drop-A} in terms of a growth condition on their tubular neighborhoods; in case of smooth inclusions  compactly contained in $ \Omega_h $, this growth condition naturally follows from an upper bound comparable to $ \frac{1}{h} $ on the norm of the shape tensor of $ \partial A $. (Indeed in Theorem~\ref{thm: minimal droplet assumption and curvature} in the appendix we investigate its validity for sets whose boundary is merely a varifold with bounded generalized second fundamental form.)

The limiting functional is defined on a product space $SBV^{2,2}_{\rm iso}(\omega) \times \mathcal{F}(\omega)$, where $SBV^{2,2}_{\rm iso}(\omega)$ is the space of `fractured and creased flat isometric immersions' and $\mathcal{F}(\omega)$ is the space of all sets of finite perimeter in $\omega$. A map $ r $ lies in $ SBV^{2,2}_{\rm iso}(\omega) $ if and only if $ r \in SBV^2(\omega, \mathbb{R}^3) $ and its approximate gradient $ \nabla r = (\partial_1 r, \partial_2 r) $ lies in $ SBV^2(\omega, \mathbb{R}^{3 \times 2})$ and satisfies the isometry condition $(\nabla r, \partial_1 r \wedge \partial_2 r ) \in SO(3) $. If $ r \in SBV^{2,2}_{\rm iso}(\omega) $ then $ r $ is a $ W^{2,2} $-isometric immersion on $ \omega \setminus \overline{(J_r \cup J_{\nabla r})} $ and consequently it is a ruled surface there (see \cite[Remark 1.2]{Pakzad}); however notice that the set $J_r \cup J_{\nabla r} $ might be dense in $ \omega $. Our limiting model thus takes the form
\begin{align*} 
	(r, D) 
	\mapsto \frac{1}{24}\int_{\omega \setminus D} Q_2(\II_r) \, \mathrm{d}x 
	+ 2 \int_{J_{(r,\nabla r)} \cap D^0} \psi_0 (\nu(J_{(r,\nabla r)})) \, \mathrm{d}\mathcal{H}^1
	+ \int_{\omega \cap \partial^* D} \psi_0 (\nu(D)) \, \mathrm{d}\mathcal{H}^1
\end{align*}
for deformations $r$ belonging to the space $SBV^{2,2}_{\rm iso}(\omega)$. 
Here $ \II_r $ is the second fundamental form of the immersion $ r $, which is computed from the approximate gradient $\nabla(\nabla r)$ of the approximate derivative of $\nabla r$. Moreover, $Q_2$ is a quadratic function derived from $D^2 W(\Id)$ and $\psi_0$ is an effective surface energy density derived from $\psi$. The elastic bending energy contribution in the bulk is given by the classical Kirchhoff plate theory in terms of a quadratic energy functional acting on the second fundamental form associated of the immersion $r$. 
There is also an obvious surface term measuring the length of the boundary $\partial^\ast D$. 
However, there are also surface terms within $\omega \setminus D$ which may arise from thin soft regions that separate parts of the matrix and whose volume vanishes asymptotically. 
Hence,  besides the limiting phase boundary $\partial^* D$, we also get surface contributions both from cracks and folds in $r(\omega \setminus D)$, corresponding to the jump sets $J_r$ and $J_{\nabla r}$ on $\omega \setminus D$, respectively, so that our limiting functional will be a Blake-Zisserman type functional, cp.\ \cite{BlakeZisserman:87,CarrieroMicheleLeaciTomarelli:96,CarrieroMicheleLeaciTomarelli:97} (also cf.\ the survey \cite{CarrieroMicheleLeaciTomarelli:15} and the references therein). 
We refer to the main result in Theorem \ref{theo:main} and to Section \ref{sec:models-results} for the precise statement and details.

We close this introduction with some comments on the technical challenges that have to be overcome. A crucial observation is that even though our minimal droplet assumption imposes only mild asymptotic regularity restrictions on the void sets, it is possible to partition the whole plate into many small cubes where the number of `bad cubes' that contain cracks is controlled. 
A geometric rigidity result then allows to conclude that on `good cubes' deformations are almost rigid. Yet, the presence of a void region impedes any compactness in Sobolev spaces and we therefore cannot proceed as in the elastic case (cp.\ \cite{FrieseckeJamesMueller:02}) in what follows. We overcome this difficulty by carefully putting together individual rigid motions on the cubes in order to construct a sequence of functions $(r_h) \subset SBV(\omega, \R^3)$ that--together with their approximate derivatives--converges to the limiting $r$ with controlled norms in $SBV$. The precise construction is, necessarily, rather involved as the envisioned energy bounds do not allow for (too much) artificial fracture as e.g.\ in a piecewise constant interpolation or (too high) artificial elastic energy as would result from mollifying within regions that will eventually be fractured.  

Furthermore there are two key steps to establish the lower bound estimates of our main theorem: For estimating the surface contribution we design an auxiliary functional which essentially tracks the surface energies along a given sequence and to which we apply a bulk relaxation result, cf.\ \cite{BraidesChambolleSolci:07,SantilliSchmidt:21}. For the bulk part the main difficulty lies in the identification of the limiting strain, more precisely, in showing that the limiting strain is asymptotically linear in the out-of-plane direction. We achieve this by considering a `flattened' plate deformation and applying an $SBV$-closure argument to such a mapping.    

For the construction of recovery sequences we first provide auxiliary explicit 3d approximations in which the deformation mappings may still be $SBV$ functions. In a secondary step these are then further approximated with the help of a bulk relaxation argument. The validity of a minimal droplet condition is then examined for arbitrary norms with the help of local estimates for the volume of tubular neighborhoods of the full crack set $ J_r \cup J_{\nabla r} \cup D $, see \eqref{eq: Minkowski content}, that amounts to require an outer Minkowski-content measurability condition (see comments and remarks after Theorem \ref{theo:main}).

\section*{Acknowledgments}

The second author was partially supported by the German Research Foundation (DFG)
within project 441138507.

\section{Notation and preliminaries}\label{sec:NotationPreliminaries}

If $ x \in \mathbb{R}^3 $ we denote its coordinates by $ x = (x_1, x_2, x_3) = (x', x_3) $, where $ x'= (x_1, x_2)  $. More generally, if $ X = (X_1 \, X_2\, X_3) $ is a $ (k \times 3) $-matrix, we write $ X' = (X_1 \, X_2) $ for the $ (k \times 2) $-matrix obtained by considering only the first two columns of $ X $. If $ v $ is a function defined on a (domain of) $ \mathbb{R}^3 $ with values in $ \mathbb{R}^k$ and $ \nabla v = (\partial_1v, \partial_2 v, \partial_3 v) $ is its differential, then we write $ \nabla' v = (\nabla v)' = (\partial_1 v, \partial_2 v) $.
	
The characteristic function of a set $ S $ is denoted by $ \rchi_S $ and the closure in $ \mathbb{R}^n $ of a set $ S  $ is denoted by $ \overline{S} $.
We use $ |\cdot | $ for the Euclidean norm and the scalar product is denoted by $ \langle \cdot, \cdot \rangle $. If $ A \subset \mathbb{R}^n $ and $ \psi $ is a norm on $\mathbb{R}^n$, we define the $ \psi $-anisotropic distance function from $ A $ by
\begin{align*}
	\dist_{\psi}(x,A) = \inf\{\psi(x-a): a \in A   \} \qquad \textrm{for $ x \in \mathbb{R}^n $}
\end{align*}
and we denote by $\psi^\circ $ the dual norm of $\psi$, i.e.
\begin{align*}
\psi^\circ(u) = \max\{ \langle u,v \rangle : \psi(v) \leq 1 \} \qquad \textrm{for $ u \in \mathbb{R}^3 $.}
\end{align*}

We say that a Borel subset $ S \subset \mathbb{R}^n $ is \emph{countably $ \mathcal{H}^{k} $-rectifiable} if there are at most countably many $C^1$ submanifolds of dimension $ k $ in $\Omega$ that cover $S$ up to an $\mathcal{H}^{k}$ negligible set. If moreover $ \mathcal{H}^{k}(S) < \infty $ then we say that $ S $ is \emph{$ \mathcal{H}^{k} $-rectifiable.} 

Suppose $ \Omega \subset \mathbb{R}^n $ is an open set and $ u \in L^1_{\rm loc}(\Omega, \mathbb{R}^m) $. We denote by $ S_u $ the \emph{approximate discontinuity set} of $ u $, see \cite[Definition 3.63]{AmbrosioFuscoPallara:00}. A point $ x \in \Omega $ is called \emph{approximate jump point} of $ u $ (see \cite[Definition 3.67]{AmbrosioFuscoPallara:00}) if there exist $ a, b \in \mathbb{R}^m $ and $ \nu \in \mathbb{S}^{n-1} $ such that $ a \neq b $ and 
\begin{equation*}
\lim_{\rho \searrow 0} \rho^{-n} \int_{B^+_\rho(x, \nu)}| u(y) -a| \, \mathrm d y = 0, \qquad 
\lim_{\rho \searrow 0} \rho^{-n} \int_{B^-_\rho(x, \nu)}| u(y) -b| \, \mathrm d y = 0.
\end{equation*}
Here $ B^+_\rho(x, \nu) = \{ y \in B_\rho(x) : \langle y-x, \nu \rangle >0   \} $ and $ B^-_\rho(x, \nu) = \{ y \in B_\rho(x) : \langle y-x, \nu \rangle <0   \} $. The triplet $(a,b , \nu) $ is uniquely determined up to a permutation of $(a,b) $ and a change of sign of $ \nu $. We denote it by $(u^+(x), u^-(x), \nu(u)(x)) $. Notice that $ J_u $ and $ S_u $ are always $ \mathcal{L}^n $-negligible subsets of $ \Omega $ and $ J_u \subset S_u $; see \cite[3.64, 3.69]{AmbrosioFuscoPallara:00}.


A function $u \in L^1(\Omega, \R^m)$ is said to lie in the space $BV(\Omega, \R^m)$ of functions of {\em bounded variation} if its distributional derivative $Du$ is a finite $\R^{m \times n}$-valued Radon measure. The total variation of $ u $ with respect to the Euclidean norm is denoted by $ |D u| $. We also need to consider \emph{the anisotropic total variation} $ \psi(Du) $ of $ D u  $ with respect to an arbitrary norm $ \psi $ for a function $ u \in BV(\Omega) $: this is the Radon measure $ \psi(D u) $ on $ \Omega $ given by 
\begin{equation*}
	\psi(D u)(B) = \int_{B} \psi \bigg( \frac{D u}{|D u|}\bigg)\, \mathrm{d}|D u| \quad \textrm{for  $ B \subset \Omega $ Borel,}
\end{equation*}
where $\frac{D u}{|D u|}$ is the $ |D u| $-measurable function satisfying $ D u = \frac{D u}{|D u|} |D u| $.

 The Federer-Volpert theorem (see \cite[Theorem 3.78]{AmbrosioFuscoPallara:00}) ensures that if $ u \in BV(\Omega, \mathbb{R}^m) $ then $ S_u $ is countably $ \mathcal{H}^{n-1} $-rectifiable, $ \mathcal{H}^{n-1}(S_u \setminus J_u) =0 $ and 
\begin{equation*}
	D u \restrict J_u = (u^+ - u^-) \otimes \nu_u \, \mathcal{H}^{n-1}\restrict J_u.
\end{equation*}
(Observe that $ | Du| (B) =0 $ if either $ \mathcal{H}^{n-1}(B) =0 $ or $ \mathcal{H}^{n-1}(B) < \infty $ and $ B \cap S_u = \emptyset $ by \cite[Lemma 3.76]{AmbrosioFuscoPallara:00}). Moreover the Calderon-Zygmund theorem (see \cite[Theorem 3.83]{AmbrosioFuscoPallara:00}) proves that $ u $ is approximately differentiable (see \cite[Definition 3.70]{AmbrosioFuscoPallara:00}) at $ \mathcal{L}^n $ a.e.\ $ x \in \Omega $ and its approximate differential $ \nabla u = (\partial_1 u, \ldots , \partial_n u) $ is the density of $ Du $ with respect to $ \mathcal{L}^n \restrict \Omega $. Denoting by $ D^s u $ the singular part of $ Du $ with respect to $ \mathcal{L}^n \restrict \Omega $, we have that 
\begin{equation*}
	Du = \nabla u \, \mathcal{L}^n \restrict \Omega + D^s u = \nabla u \, \mathcal{L}^n \restrict \Omega + (u^+ - u^-)\otimes \nu(u) \, \mathcal{H}^{n-1}\restrict J_u + D^c u,
\end{equation*}
where $ D^c u = D^s u \restrict (\Omega \setminus S_u) $. We say that $ u $ lies in the space $ SBV(\Omega, \mathbb{R}^m) $ of \emph{special functions of bounded variations} if and only $ D^c u =0 $. For $ p \geq 1 $ we write
\begin{equation*}
	SBV^p(\Omega, \mathbb{R}^m) = \{ u \in SBV(\Omega, \mathbb{R}^m) : \nabla u \in L^p(\Omega, \mathbb{R}^m) ~\textrm{and}~ \mathcal{H}^{n-1}(S_u) < \infty   \}.
\end{equation*}

If $E \subset \R^n$ is a Borel subset then we define for $ 0 \leq t \leq 1 $ the set
\begin{equation*}
	E^t = \{ x \in \R^n : \lim_{\rho \searrow 0} \rho^{-n} \mathcal{L}^n \big( B_{\rho}(x) \setminus E \big) =  t\} 
\end{equation*} 
and we set $ \partial^\ast E = \mathbb{R}^{n} \setminus (E^0 \cup E^1) $. The set $ \partial^\ast E $ is the \emph{essential boundary of $ E $.}  If $ \Omega\subset \mathbb{R}^n  $ is open then we say that \emph{$ E $ has finite perimeter in $ \Omega $} if and only if $ \mathcal{H}^{n-1}(\Omega \cap \partial^\ast E) < \infty $, or equivalently if and only if $ \rchi_E \in BV(\Omega) $ (notice that $ \chi_E \in SBV(\Omega) $); see \cite[3.5]{AmbrosioFuscoPallara:00} for details. If $ E $ has finite perimeter in $ \Omega $ we define $ \nu(E) : J_{\chi_E} \rightarrow \mathbb{S}^{n-1} $ by
$$ \nu(E) = -  \nu(\chi_E).$$
Notice that $ J_{\chi_E} \subset S_{\chi_E} = \Omega \cap \partial^\ast E $, $ \mathcal{H}^{n-1}(\Omega \cap \partial^\ast E \setminus J_{\chi_E}) =0 $ and $ \nu(E) $ is the \emph{measure-theoretic exterior unit-normal to $ E $}. We write $\mathcal{F}(\Omega)$ to denote the collection of sets $E \subset \Omega$ of finite perimeter in $\Omega$. Moreover if $ E \in \mathcal{F}(\Omega) $ and $ \psi $ is an arbitrary norm then we notice that
\begin{equation*}
	\psi(D \rchi_E)  =
	 \psi(\nu(E)) \mathcal{H}^{n-1} \restrict \partial^\ast E.
\end{equation*}	
and the following \emph{anisotropic coarea formula for $ BV $ functions} holds: if $ u \in BV(\Omega) $, then $ \{u > t\} $ has finite perimeter in $ \Omega $ for $ \mathcal{L}^1 $ a.e.\ $ t \in \mathbb{R} $ and
\begin{align}\label{eq: Coarea}
	\psi(D u)(B) = \int_{-\infty}^{+ \infty} \psi(D \rchi_{\{u > t    \}}) (B)\, \mathrm{d}t
\end{align}
for each Borel subset $ B \subset \Omega $ (see \cite[Theorem 3.40]{AmbrosioFuscoPallara:00} and \cite[page 5, eq.\ (1)]{SantilliSchmidt:21}).

\begin{remark}\label{rmK: coordinate changes SBV}
If $ \varphi : \Omega \rightarrow \Omega' $ is a bijective Lipschitz map with Lipschitz inverse and $ u \in SBV(\Omega', \mathbb{R}) $, then  $ u \circ \varphi \in SBV(\Omega, \mathbb{R}) $, $ S_{u \circ \varphi} = \varphi^{-1}(S_u) $ and the approximate gradient of $ u \circ \varphi $ satisfies the equation
\begin{equation}\label{eq:chain-rule-diffeo}
	\nabla (u \circ \varphi)(x) = \nabla u(\varphi(x)) \nabla \varphi(x) \qquad \textrm{for $ \mathcal{L}^n $ a.e.\ $ x \in \Omega $;}
\end{equation}
see \cite[exercise 4.5 at page 252]{AmbrosioFuscoPallara:00}. Additionally, one can check that $ J_{u \circ \varphi} = \varphi^{-1}(J_u) $ with
\begin{align}\label{eq:normal-diffeo}
	\nu(u \circ \varphi)(x) = \frac{\nabla \varphi(x)^T(\nu_u(\varphi(x)))}{|\nabla \varphi(x)^T(\nu_u(\varphi(x)))|} 
    \quad\text{and}\quad 
	(u \circ \varphi)^\pm(x) = u^\pm(\varphi(x))
\end{align}
for every $ x \in J_{u \circ \varphi} $.
\end{remark}

\section{Models and main results}\label{sec:models-results}

Let $ \omega \subset \mathbb{R}^2 $ be a bounded open domain with Lipschitz boundary. For $h > 0$ we define the reference configuration of a thin plate by 
\begin{equation*}
\Omega_h = \omega \times (0, h). 
\end{equation*}
We fix a stored energy density $ W : \mathbb{R}^{3\times 3} \rightarrow [0, +\infty) $ for the elastic matrix which satisfies the following standard set of assumptions: 
\begin{itemize} 
\item frame indifference: $W(RX) = W(X)$ for all $ X \in \mathbb{R}^{3 \times 3} $, $ R \in \SO(3)$, 
\item regularity: $W$ is $C^2$ in a neighborhood of $\SO(3)$, 
\item single well and growth: $W(\Id) = 0$ and there are constants $c, C > 0$ such that 
\begin{equation*}
  c \dist^2(X, \SO(3)) 
  \leq W(X) 
  \leq C (1 + |X|^2)
  \quad \text{for every $ X \in \mathbb{R}^{3 \times 3} $.}
\end{equation*}
\end{itemize}
We remark that, as a consequence, the quadratic form $Q_3$ on $\R^{3 \times 3}$ of linearized elasticity, given by $Q_3(X) = D^2 W(\Id)[X,X]$ satisfies the linear frame indifference relation 
\begin{equation*}
  Q_3(X+A)=Q_3(X)
  \quad\text{for all $ X, A \in \mathbb{R}^{3 \times 3} $ with $ A = - A^T $}. 
\end{equation*}
Finally we fix a norm $ \psi $ on $\mathbb{R}^3$ and we define the energy functionals 
\begin{equation*}
\mathcal{J}_h: W^{1,2}(\Omega_h, \mathbb{R}^3) \times \mathcal{F}(\Omega_h) \rightarrow \mathbb{R}
\end{equation*}
by
\begin{align}\label{eq:E-unrescaled}
	\mathcal{J}_h(v,A) = \int_{\Omega_h \setminus A} W(\nabla v) \, \mathrm{d}x +  h^{2} \int_{\Omega \cap \partial^\ast A} \psi(\nu(A))\, \mathrm{d}\mathcal{H}^{2}.
\end{align}

From now on we denote the tubular neighborhood of width $r > 0$ of a set $A \subset \mathbb{R}^3 $ with respect to the fixed norm $ \psi $ as
\begin{align*}
 A^{(r)} =  \{ x \in \mathbb{R}^3 : \operatorname{dist}_{\psi^\circ}(x, A) \le r \}.
\end{align*}
We say that a family $(A_{h})_{h \in (0,1)} $ such that $ A_{h} \in \mathcal{F}(\Omega_{h}) $ satisfies \emph{the $ \psi $-minimal droplet assumption in $\Omega_h $} if and only if there exists an increasing positive function $\zeta: (0,2) \rightarrow \mathbb{R}$ with $ \lim_{t \to 0+} \zeta(t) =0 $ such that
\begin{align}\label{eq:min-drop-A}
   \mathcal{L}^3 \big( (A_h^{(sh)} \setminus A_h) \cap  \Omega_h^- \big) 
  &\le (1 +  \zeta(h + s)) s h \int_{\Omega_h \cap \partial^\ast A_h} \psi( \nu(A_h) ) \, \mathrm d\mathcal{H}^2  + \zeta(h + s) s h^2
\end{align}
for all $s \in (0,1)$,  where $\Omega^-_h = \{x \in \Omega_h : \dist_{|\cdot|}(x', \partial \omega) > h\}$. 

\begin{remark}
We remark that this condition can be interpreted as a (diverging) bound on the curvature of boundary of the sets. Indeed, in the isotropic setting for a sequence of smooth open sets $A_h \subset \Omega_h $ satisfying $ \overline{A_h} \subset \Omega_h $, the minimal droplet assumption \eqref{eq:min-drop-A} holds if the second fundamental form $ \II_{\partial A_h} $ of $ \partial A_h $ satisfies $ | \II_{\partial A_h}(x) | \leq Ch^{-1} $ for every $ x \in \partial A_h $. 

The smoothness assumption is not necessary; indeed in the appendix we prove the minimal droplet assumption (in the Euclidean setting) for sets of finite perimeter $ A \in \mathcal{F}(\Omega_h) $ such that the first variation in $ \Omega_h $ (in the sense of varifold's theory) of the essential boundary $ \partial^\ast A $ is absolutely continuous with respect to $ \mathcal{H}^n \restrict \partial^\ast A $ and the generalized second fundamental form is controlled by $  \frac{1}{h} $.
\end{remark}

\begin{remark}
If $ \psi $ is a uniformly convex $ C^2 $-norm on $ \mathbb{R}^3 $, employing the classical notions of $ \psi $-anisotropic mean curvature $ H^\psi_M $  and $ \psi $-anisotropic Gaussian curvature $ K^\psi_M $ for a $ C^2 $-hypersurface $ M $ (see for instance \cite[Definition 2.26]{DeRosaKolasinskiSantilli}), it is still true (and this can be proved employing the anisotropic disintegration formula in \cite[Theorem 3.18]{HugSantilli}) that the $ \psi $-minimal droplet assumption holds for a sequence of smooth open sets $ A_h \subset \Omega_h $ satisfying $ \overline{A_h} \subset \Omega_h $ and the bounds
	\begin{equation*}
	|H^\psi_{\partial A_h}(x)| \leq \frac{C}{h} \quad \textrm{and} \quad  |K^\psi_{\partial A_h}(x)| \leq \frac{C}{h^2} \quad \textrm{for $ x \in \partial A_h $.}
	\end{equation*}
\end{remark}

In order to formulate the problem in the fixed domain $ \Omega := \Omega_1 $ we change coordinates according to $y(x) = v(x', hx_3)$ for $ x \in \Omega $, $D_h = \operatorname{diag}(1,1,h^{-1}) A_h$. We also define the rescaled gradient operator $ \nabla_h $ as
\begin{equation*}
	\nabla_h y = (\nabla' y, h^{-1} \partial_3 y)
\end{equation*}
and the rescaled exterior unit-normal $ \nu_h $ as
\begin{equation*}
\nu_h(A) = (\nu'(D), h^{-1} \nu_3(D)).
\end{equation*}

We denote the rescaled tubular neighborhoods by $D^{(r)_h} = \operatorname{diag}(1,1,h^{-1}) A^{(r)}$ and say that $(D_h)_{h>0} \subset \mathcal{F}(\Omega)$ satisfies the (rescaled) minimal droplet assumtion if $A$ satisfies \eqref{eq:min-drop-A}, i.e., if 
\begin{align}\label{eq:min-drop-D}
  \mathcal{L}^3 \big( (D_h^{(sh)_h} \setminus D_h) \cap \Omega^- \big) 
  \le (1 + \zeta(h+s)) sh \int_{\partial^* D_h \cap \Omega} \psi(\nu_h(D_h)) \, \mathrm d\mathcal{H}^2 + \zeta(h+s) s h,
\end{align}
where $\Omega^- = \{x \in \Omega_1 : \dist_{|\cdot|}(x', \partial \omega) > h\}$. We consider the rescaled functionals $$\mathcal{E}_{h} : W^{1,2}(\Omega, \mathbb{R}^3) \times \mathcal{F}(\Omega) \rightarrow  \mathbb{R}$$ defined by
\begin{align*}
	\mathcal{E}_{h}(y,D) = h^{-2} \int_{\Omega \setminus D} W(\nabla_h y) \, \mathrm{d}x  
   + \int_{\Omega \cap \partial^* D} \psi (\nu_h(D)) \, \mathrm{d}\mathcal{H}^2.  
\end{align*}
 We extend these functionals to $L^1(\Omega, \mathbb{R}^3) \times \mathcal{F}(\Omega)$ by setting 
$$\mathcal{E}_{h}(y,D) = +\infty \qquad \textrm{if $(y, D) \notin W^{1,2}(\Omega, \mathbb{R}^3) \times \mathcal{F}(\Omega)$}.$$

With a slight abuse of notation, in the following we do not distinguish between functions defined on $\omega$ and functions on $\Omega$ that do not depend on $x_3$. In particular, we make the identifications 
\begin{align*}
   SBV^2(\omega, \mathbb{R}^N) 
   &= \big\{ u \in SBV^2(\Omega, \mathbb{R}^N): \partial_3 u =0, \; \nu_3(u)  =0 \big\}, \\ 
   \mathcal{F}(\omega) 
   &= \big\{ D \in \mathcal{F}(\Omega): \nu_3(D) =0 \big\}.
\end{align*}
Our limiting deformations $r$ turn out to be isometric immersions away from the jump set $J_{(r,\nabla r)} = J_r \cup J_{\nabla r}$. To describe them we introduce the space 
\begin{align*}
  SBV^{2,2}_{\rm iso}(\omega)  
  &= \big\{ r \in SBV^2(\omega, \mathbb{R}^3) \cap L^{\infty}(\omega, \mathbb{R}^3) : \\ 
  &\qquad \nabla r \in SBV^2(\omega, \mathbb{R}^{3 \times 2}),(\nabla r, \partial_1 r \wedge \partial _2 r )\in \SO(3) \text{ a.e.} \big\} 
\end{align*}
To a mapping $r \in SBV^{2,2}_{\rm iso}(\omega)$ we associate its second fundamental form as being given by 
$$ \II(x') 
   = - \big( \partial_{ij} r \cdot (\partial_1 r \wedge \partial _2 r) \big)_{1 \le i,j \le 2}. $$ 

We also introduce the condition 
\begin{align}\label{eq: Minkowski content}
	\begin{split}
		&2 \int_{J_{(r,\nabla r)} \cap D^0} \psi_0 (\nu(J_{(r,\nabla r)})) \, \mathrm{d}\mathcal{H}^1
		+ \int_{\omega \cap \partial^* D} \psi_0 (\nu(D)) \, \mathrm{d}\mathcal{H}^1 \\ 
		&~~= \liminf_{t \to 0} t^{-1} \mathcal{L}^2 \big( \omega \cap \{ x \in \mathbb{R}^2 : \operatorname{dist}_{\psi_0^\circ}(x, J_{(r,\nabla r)} \cup D) \le t \} \setminus D \big),
	\end{split}
\end{align}
for $(r, D) \in SBV^{2,2}_{\rm iso}(\omega) \times \mathcal{F}(\omega)$. It can be interpreted as an `outer Minkowski-content measurability' property of the set $J_{(r,\nabla r)} \cup D$. In particular this is satisfied provided $ \mathcal{H}^1\big(\partial \omega \cap \partial(J_{(r,\nabla r)} \cup D)  \big) =0 $, $\partial(J_{(r,\nabla r)} \cup D)$ is $\mathcal{H}^1$-rectifiable and 
\begin{align*}
	\mu(B_r(x)) \ge \gamma r 
	\quad \forall\, x \in \partial(J_{(r,\nabla r)} \cup D)~ \forall r \in (0,1) 
\end{align*}
for a constant $\gamma > 0$ and a finite measure $\mu$ on $\R^2$ with $\mu \ll \mathcal{H}^1$, cp.\ Theorem~\ref{theo:Minkowski-density-lb}. 

Our main result is the following $\Gamma$-convergence type result towards a limiting `Blake-Zisserman-Kirchhoff' plate functional 
\begin{align*}
  \mathcal{E}(r, D) 
  &= \frac{1}{24}\int_{\omega \setminus D} Q_2(\II) \, \mathrm{d}x' \\ 
   &\quad + 2 \int_{J_{(r,\nabla r)} \cap D^0} \psi_0 (\nu(J_{(r,\nabla r)})) \, \mathrm{d}\mathcal{H}^1
   + \int_{\omega \cap \partial^* D} \psi_0 (\nu(D)) \, \mathrm{d}\mathcal{H}^1
\end{align*}
for $(r, D) \in SBV^{2,2}_{\rm iso}(\omega) \times \mathcal{F}(\omega)$ and $= + \infty$ elsewhere on $L^1(\Omega, \mathbb{R}^3) \times \mathcal{F}(\Omega)$. Here the relaxed quadratic form $Q_2$ on $\R^{2\times 2}$ and relaxed norm $\psi_0$ on $\R^2$ are given by 
\begin{align*}
 Q_2(X) 
   &= \min_{c \in \R^3} \begin{pmatrix}[1.1] 
      X & \vline & c' \\ 
     \hline 
     0 & \vline & c_3   
     \end{pmatrix}
    = \min_{c \in \R^3} \begin{pmatrix}[1.1] 
      X & \vline & c' \\ 
     \hline 
     (c')^T & \vline & c_3   
     \end{pmatrix} 
\shortintertext{and} 
 \psi_0(x_1,x_2) 
   &= \min_{c \in \R} \psi(x_1,x_2,c).
\end{align*}

The factor 2 multiplying the jump part arises from cracks outside of $D$ can develop as a result of asymptotically thin tubular neighborhoods of $J_{(r,\nabla r)}$ whose boundary is asymptotically twice as big as the surface area of $J_{(r,\nabla r)} \cap D^0$.

\begin{theorem}\label{theo:main} \begin{itemize}
\item[(i)] Let $(r, D) \subset L^\infty(\Omega, \mathbb{R}^3) \times \mathcal{F}(\Omega)$. A sequence $(y_h, D_h) \subset L^\infty(\Omega, \mathbb{R}^3) \times \mathcal{F}(\Omega)$  verifies  
\begin{align*}
  \liminf_{h \to 0} \mathcal{E}_{h}(y_h, D_h) 
  \ge \mathcal{E}(r, D), 
\end{align*}
provided $ y_h \to r$ in $L^1(\Omega, \mathbb{R}^3)$, $\rchi_{D_h} \to \rchi_D$ in $L^1(\Omega, \mathbb{R})$ with $\limsup_{h} \| y_h \|_{L^\infty} < \infty$ and $ (D_h)_{0 < h < 1} $ satisfying the $ \psi $-minimal droplet assumption. 
\item [(ii)] For any $(r, D) \subset L^\infty(\Omega, \mathbb{R}^3) \times \mathcal{F}(\Omega)$ there exists a sequence $(y_h, D_h) \subset L^\infty(\Omega, \mathbb{R}^3) \times \mathcal{F}(\Omega)$ with $y_h \to r$ in $L^1(\Omega, \mathbb{R}^3)$ and $\rchi_{D_h} \to \rchi_D$ in $L^1(\Omega, \mathbb{R})$ such that 
\begin{align*}
  \lim_{h \to 0} \mathcal{E}_{h}(y_h, D_h) 
  = \mathcal{E}(r, D). 
\end{align*}
Moreover, there is a universal constant $c_0$ such that $\| y_h \|_{L^\infty} \le \| y \|_{L^\infty} + c_0 h$. 
There is also a universal error function $\zeta_0$ such that for any $(r, D) \in SBV^{2,2}_{\rm iso}(\omega) \times \mathcal{F}(\omega)$ that satisfy \eqref{eq: Minkowski content} the sets $D_h$ can be chosen such that $(D_h)_{0 < h < 1}$ satisfies the $ \psi $-minimal droplet assumption with $ \zeta_0 $.
\end{itemize}
\end{theorem}

As a consequence to our compactness Theorem \ref{Compactness} below we also have that any sequence $(y_h, D_h) \subset L^\infty(\Omega, \mathbb{R}^3) \times \mathcal{F}(\Omega)$ satisfying \eqref{eq:min-drop-D} and 
\begin{align*}
  \limsup_{h \to 0} \big( \mathcal{E}_{h}(y_h, D_h) + \| y_h \|_{L^\infty} \big) 
  < \infty. 
\end{align*}
has a subsequence (not relabeled) such that $\chi_{D_h} \to \chi_{D}$ in $L^1(\Omega)$ and $\chi_{\Omega \setminus D_h} y_h \to \chi_{\omega \setminus D}r$ in $L^1(\Omega,\R^3)$. This complements the convergence Theorem~\ref{theo:main} as $\mathcal{E}_{h}(y_h, D_h)$ and $\mathcal{E}(r, D)$ indeed only depend on $(\chi_{D_h} y_h, D_h)$ and $(\chi_{D} r, D)$, respectively. 

\begin{remark}
\begin{enumerate}
\item The constant $c_0$ and error function $\zeta_0$ are universal in the sense that they only depend on the norm $\psi$. In particular, they are independent of the specific limiting configuration $(r,D)$. Our working with $L^\infty$ bounded deformations can thus be justified by assuming the plate to be restricted to a bounded laboratory by requiring $\| y_h \| \le M + c_0 h$, $M > 0$ a given constant. This constraint can be energetically enforced by (re-)defining $\mathcal{E}_{h}(y, D) = \infty$ if $\| y \| > M + c_0 h$ and $\mathcal{E}(r, D) = \infty$ if $\| r \| > M$. Likewise the minimal droplet assumption \eqref{eq:min-drop-D} can be installed energetically by moreover requiring that $\mathcal{E}_{h}(y, D) = \infty$ if \eqref{eq:min-drop-D} is violated for $\zeta_0$.

\item Instead of surface energies scaling with $h^2$ in \eqref{eq:E-unrescaled} one may look into more general scalings $h^{\alpha}$. For $\alpha < 2$ one obtains a purely elastic Kirchhoff plate theory in the limit $h \to 0$ as then our results show that bounded energy sequences cannot develop nontrivial voids, cracks or folds in the limit. Conversely, if $\alpha < 2$ one would arrive at a trivial theory as there may be infinitely long cracks at zero energy in the limit. 

In order to arrive at an interesting limit functional in this case, one is lead to rescale the bulk part by $h^{-\alpha-1}$ insead of $h^{-3}$ and to consider a different plate theory than Kirchhoff's theory, cf.\ \cite{FrieseckeJamesMueller:06}.

\item We remark that Theorem~\ref{theo:main} raises two interesting questions which are beyond the scope of this contribution: First, one might wonder if the minimal droplet assumption \eqref{eq:min-drop-D} in (i) and in the Compactness Theorem~\ref{Compactness} can be dropped. In view of recent results in \cite{FriedrichKreutzZemas:21} it seems plausible that this condition might be further weakened. Yet, in order to understand if such a condition could be dropped completely, as in the two-dimensional seeting \cite{Schmidt:17}, would require an extension of the quantitative rigidity results in \cite{FriedrichSchmidt:15} to the three-dimensional setting which appears highly non-trivial. Second, it would be interesting to investigate if for general $(r, D) \in  SBV^{2,2}_{\rm iso}(\omega) \times \mathcal{F}(\omega)$ recovery sequences $(y_h, D_h)$ with $D_h$ obeying the minimal droplet property \eqref{eq:min-drop-D} can be constructed. This would follow from Theorem ~\ref{theo:main} if one knew that the space of configurations $(r,D)$ satisfying \eqref{eq: Minkowski content} is dense in $SBV^{2,2}_{\rm iso}(\omega) \times \mathcal{F}(\omega)$. 
\end{enumerate}
\end{remark}

\section{Compactness}

Our first aim is to prove compactness of bounded energy sequences $(y_h, D_h)$. To this end, we consider a modification $D'_h \supset D_h$ of $D_h$ which is obtained from a suitable covering by cubes and which satisfies $\mathcal{L}^3(D'_h \setminus D_h) \le Ch$. 

Observe that, if $(A_h)$ satisfies the minimal droplet assumption, for each given $s \in (0, \frac{1}{3})$ we may choose $s_h = s_h(A_h) \in (s^2, s)$ and set $\sigma_h = s_h h$ such that 

\begin{align}\label{eq:A3sh-surf-est}
  \int_{\Omega_{h}^- \cap \partial^* A_h^{(3\sigma_h)}} \psi \big( \nu(A_h^{(3\sigma_h)}) \big) \, \mathrm d\mathcal{H}^2 
  \le \frac{1 + \zeta(h + 3s)}{1 - s} \int_{\partial^* A \cap \Omega_{h}} \psi( \nu(A_h) ) \, \mathrm d\mathcal{H}^2 
  + \frac{\zeta(h + 3s)}{1 - s}.  
\end{align}
To see this note that 
\begin{align*}
  &3 s h \zeta(h + 3s) + 3 s h (1 + \zeta(h + 3s)) \psi\big(D \chi_{A_h}\big)(\Omega_h) \\ 
  &\quad\geq \mathcal{L}^3 \big( (A_h^{(3sh)} \setminus A_h) \cap \Omega_h^- \big) \\ 
  &\quad= \int_{(A_h^{(3sh)} \setminus {A_h}) \cap \Omega_h^-}\psi \big( \nabla \operatorname{dist}_{\psi^\circ}(\cdot, A_h) \big)\, \mathrm{d}x \\
  &\quad= \int_{0}^{3sh}  \psi\big(D \chi_{A_h^{(t)}}\big)(\Omega_h^-)\,\mathrm{d}t\\
  &\quad\ge h \int_{3s^2}^{3s} \psi\big(D \chi_{A_h^{(th)}}\big)(\Omega_h^-)\,\mathrm{d}t \\ 
  &\quad\ge 3 s h (1 - s) \psi\big(D \chi_{A_h^{(3s_h h)}}\big)(\Omega_h^-) 
\end{align*}
for some $s_h$ such that $3s_h \in (3s^2, 3s)$. 

We can now define our covering cubes. Given $(v_h, A_h)$, respectively their rescalings $(y_h, D_h)$, and fixing an $s \in (0, 1)$ we classify the elements of $\{ a' \in h \Z^2 : ( a' + (0,h)^2 ) \cap \omega \ne \emptyset \}$ and corresponding cubes $Q_h(a') = (a',0) + (0,h)^3$ and rescaled cuboids $\hat{Q}_{h}(a') = (a',0) + (0,h)^2 \times (0,1)$ based at these points into the following types: 
\begin{itemize}
\item If $Q_h(a') \subset \Omega_h$ and $\overline{Q_h(a')} \cap \partial A_h^{(\sigma_h)} = \emptyset$, then we write $a' \in \mathcal{G}_{h}$ and say that the cuboids $Q_h(a')$ and $\hat{Q}_h(a')$ are \emph{good}. We also set $\mathcal{G}_{h}^{\rm v} = \{a' \in \mathcal{G}_{h} : Q_h(a') \subset A_h^{(\sigma_h)} \}$, $\mathcal{G}_{h}^{\rm el} = \mathcal{G}_{h} \setminus \mathcal{G}_{h}^{\rm v}$. 
\item In case $Q_h(a') \subset \Omega_h$ and $\overline{Q_h(a')} \cap \partial A_h^{(\sigma_h)} \ne \emptyset$ we say that $Q_h(a')$ and $\hat{Q}_h(a')$ are interior bad cuboids and write $a' \in \mathcal{B}_{h}^{\rm in}$. 
\item If $a' \in \mathcal{B}_{h}^{\rm in}$ or $Q(a') \cap \partial \Omega_h \ne \emptyset$, we write $a' \in \mathcal{B}_{h}$ and call $Q_h(a')$ and $\hat{Q}_h(a')$ bad cuboids. 
\end{itemize}
We then set 
\begin{align*}
  A_h' 
   = \operatorname{diag}(1,1,h) D_h'
   = \bigcup_{a' \in \mathcal{G}_h^{\rm v} \cup \mathcal{B}_h} \overline{Q_h(a')}.  
\end{align*}

\begin{theorem}\label{Compactness}
Suppose $(y_h, D_h)$ is a sequence in $W^{1,2}(\Omega, \mathbb{R}^3)  \times \mathcal{F}(\Omega)$ with 
\begin{align}
  \limsup_{h \to 0} \big( \mathcal{E}_{h}(y_h, D_h) + \| y_h \|_{L^\infty} \big) 
  < \infty \label{eq:E-y-bound}
\end{align}
and $ (D_h)_{0 < h < 1} $ satisfying the $ \psi $-minimal droplet assumption. Then there exists $ (r, D) \in SBV^{2,2}_{\rm iso}(\omega) \times \mathcal{F}(\omega) $ such that, up to subsequences,
\begin{align} 
\label{Compactness:2}
\mathcal{L}^3(D_h' \triangle D_h) 
&\rightarrow  0, \\
\label{Compactness:3}
\chi_{\Omega \setminus D_h'} y_h 
&\rightarrow \chi_{\omega \setminus D}r \quad \textrm{in $\Lp{2}$}, \\
\label{Compactness:4}
\chi_{\Omega \setminus D_h'} \nabla_h y_h 
&\rightarrow \chi_{\omega \setminus D} ( \nabla' r, \partial_1 r \wedge \partial_2 r ) \quad \textrm{in $\Lp{2}$}.
\end{align}
\end{theorem}
Here the terms on the right hand sides are to be interpreted as functions of $x = (x', x_3) \in \Omega$ that only depend on $x' \in \omega$.

\begin{remark}\label{rmk:sym-second-derivatives}
If $r \in SBV^{2,2}_{\rm iso}(\omega)$ is such that $\mathcal{H}^1(J_r) < \infty$, it follows from \cite[Theorem~3.1]{ChambolleGiacominiPonsiglione:07} that the measure
\begin{equation*}
\curl(\nabla r_{i}) = 
\begin{bmatrix}
0       & \Der_{1}(\partial_{2}r_{i}) - \Der_{2}(\partial_{1}r_{i})  \\
 \Der_{2}(\partial_{1}r_{i}) - \Der_{1}(\partial_{2}r_{i})      &  0 \\
\end{bmatrix} \in \mathcal{M}(\omega, \R^{2 \times 2}) 
\end{equation*}
is absolutely continuous with respect to $\mathcal{H}^{1} \restrict J_{r_{i}}$ for $i = 1, 2, 3$. Therefore, 
\begin{equation*}
\partial_{1}(\partial_{2}r)(x) = \partial_{2}(\partial_{1}r)(x) \quad \textrm{for $ \mathcal{L}^{2} $ a.e.\ $ x \in \omega $.}
\end{equation*}
\end{remark}

\begin{proof} \noindent {\bf The number of bad cubes.}
We proceed to show that the number of bad cubes is bounded by $C h^{-1}$. Set $c_1 = \min \psi/|\cdot|$. We first prove that there is a constant $c = c(s) > 0$ such that for each $x \in \partial A_h^{(\sigma_h)} \cap \overline{\Omega_h}$ with $\dist(x', \partial \omega) > 2 c_1 h$ we have 
\begin{align}\label{eq:Ash-A-volume}
 \mathcal{L}^3 \big( ( A_h^{(\sigma_h)} \setminus A ) \cap Z_{h}(x') \big) \ge c h^3, 
\end{align}
where $Z_{h}(x')$ is the cylinder $Z_{h}(x') = \{ y \in \mathbb{R}^3 : |y' - x'| \le 2 c_1 h, \ 0 \le y_3 \le h \}$. 

Indeed, for such an $x$ there is a $y \in \overline{A_h} \cap Z_{h}(x')$ with $\psi(x - y) = \sigma_h$ and the portion 
\[ B 
   = \big\{ z \in Z_{h}(x') : \psi(z-x) < \sigma_h \text{ and } \psi(z-y) < \sigma_h \big\} \] 
of the intersection of the open $\psi$-balls around $x$ and $y$ of radius $\sigma_h$ within $Z_{h}(x')$ satisfies 
\begin{align*}
  B \subset (A_h^{(\sigma_h)} \setminus A ) \cap Z_{h}(x') 
  \quad\text{and}\quad  
 \mathcal{L}^3 ( B ) \ge c_2 (s^2 h)^3 
\end{align*}
for a constant $c_2$ only depending on $\psi$. 

Let now $\mathcal{B}_{h}' = \{ a' \in \mathcal{B}_{h} : \dist(a', \partial \omega) > (\sqrt{2}+2 c_1)h \}$. For each $a' \in \mathcal{B}_{h}$ fix an $x(a') \in \partial A_h^{(\sigma_h)} \cap \overline{Q_h(a')}$. Using \eqref{eq:Ash-A-volume} we get that 
\begin{align*} 
    \#  \mathcal{B}_{h}'
    &\le c^{-1} h^{-3} \sum_{a' \in \mathcal{B}_{h}'} \mathcal{L}^3 \big( ( A_h^{(\sigma_h)} \setminus A ) \cap Z_{h}(x'(a')) \big) \\ 
    &\le C h^{-3} \mathcal{L}^3 \big( ( A_h^{(\sigma_h)} \setminus A ) \cap \Omega_h^- \big)  \\ 
    &\le C h^{-2} s (1 + \zeta(h + s)) \int_{\partial^* A_h \cap \Omega_h} \psi(\nu(A_h)) \, \mathrm d\mathcal{H}^2 
    + C h^{-1} s (1 + \zeta(h + s)), 
\end{align*}
where we have used \eqref{eq:min-drop-A} and $h^{-1} \sigma_h = s_h \le s$ and $\zeta(h + s_h) \le \zeta(h + s)$.  
As a direct consequence of the energy bound and the fact that $\partial \omega$ is Lipschitz we thus get  
\begin{equation}\label{number of bad cubes}
  \# \mathcal{B}_{h} 
  \leq \#\mathcal{B}_{h}' + C \mathcal{H}^1(\partial \omega) h^{-1}  
  \leq C h^{-1}.
\end{equation}

As by \eqref{eq:min-drop-D} and the energy bound also 
\begin{align}\label{eq:vol-sh-tube}
  \mathcal{L}^3(D_h^{(sh)_h} \setminus D_h) 
  \le Ch, 
\end{align}
we obtain from  
\begin{align}\label{eq:Dhprime-Dh}
  D_h' \triangle D_h 
  = D_h' \setminus D_h 
  \subset \bigcup_{a' \in \mathcal{B}_h} \overline{Q_h(a' )}\cup (D_h^{(sh)_h} \setminus D_h), 
\end{align}
the estimate $\mathcal{L}^3(D_h' \triangle D_h) \le C h$ and hence \eqref{Compactness:2} follows. 

\medskip

\noindent {\bf Estimates on a good cube.} Fix $a' \in \mathcal{G}_{h}^{\rm el}$. We use the rigidity theorem \cite[Theorem 3.1]{FrieseckeJamesMueller:02} to select a constant $C$ and a rotation $R_{h,a'} \in \SO(3)$ such that
\begin{equation*}
  \int_{Q_h(a')} |\nabla v_h - R_{h,a}|^{2} \, \mathrm dx 
  \leq C \int_{Q_h(a')} \dist^2(\nabla v_h, \SO(3)) \, \mathrm dx 
\end{equation*}
and, by rescaling, we obtain
\begin{equation}\label{estimate on one cube I}
  \int_{\hat{Q}_{h}(a')} |\nabla_h y_h - R_{h,a}|^{2} \, \mathrm dx 
  \leq C \int_{\hat{Q}_{h}(a')} \dist^2(\nabla_h y_h, \SO(3)) \, \mathrm dx. 
\end{equation}
We define the piecewise constant function $R_{h}: \omega \rightarrow \SO(3)$ so that 
\begin{equation}\label{eq:Rh-def}
  R_{h}(x') 
  = \begin{cases} 
	     R_{h,a} & \text{if } x' \in a' + (0,h)^2 \text{ with } a' \in \mathcal{G}_{h}^{\rm el}, \\ 
       \Id & \textrm{otherwise}. 
    \end{cases}
\end{equation}
Let $\chi_{h}^{\rm el} = \chi_{\Omega \setminus D_h'}$ be the characteristic function of $\bigcup_{a' \in \mathcal{G}_h^{\rm el}} \hat{Q}_h$. It follows from \eqref{estimate on one cube I} and the energy estimate \eqref{eq:E-y-bound} that 
\begin{equation}\label{L2 distance of the rescaled gradients from piecewice rotational function:2}
  \int_{\Omega} \chi_{h}^{\rm el}(x) |\nabla_{h} y_h(x) - R_{h}(x')|^{2} \, \mathrm dx 
	\leq C \int_{\Omega} \chi_{h}^{\rm el} \dist^2(\nabla_h y_h, \SO(3)) \, \mathrm dx 
  \leq C h^2. 
\end{equation}

We define 
\begin{equation*}
  c_{h,a'} 
  = \dashint_{Q_{h}(a')} ( v_h(x) - R_{h,a'} x ) \, \mathrm dx 
\end{equation*}
and the linear map $r_{h,a'} : \R^{3} \rightarrow \R^{3}$ by
\begin{equation*}
  r_{h,a'}(x) = R_{h,a'}(x', h x_3) + c_{h,a'}.
\end{equation*}
Note that the $c_{h,a'}$ are uniformly bounded since $\| y_h \|_{\infty} \leq C$. Thus we also have $\| r_{h,a'} \|_{\infty,\omega} \le C$ for a suitable constant $C > 0$. We apply the Poincarè-Wirtinger inequality to estimate
\begin{flalign*}
  & h^{-2} \int_{\hat{Q}_{h}(a')} |y_h - r_{h,a'}|^{2} \, \mathrm dx 
    = h^{-3} \int_{Q_h(a')} |v_h(x) - r_{,a'}(x',hx_3)|^{2} \, \mathrm dx \\
  &\qquad \leq C h^{-1} \int_{Q_h(a')} |\nabla v_h - R_{h,a'}|^{2} \, \mathrm dx
    = C \int_{\hat{Q}_{h}(a')} |\nabla_h y_h - R_{h,a'}|^{2} \, \mathrm dx.
\end{flalign*}
Applying \eqref{estimate on one cube I} we infer
\begin{flalign}\label{estimate on one cube II}
  \int_{\hat{Q}_{h}(a')} |y_h - r_{h,a'}|^{2} \, \mathrm dx 
  \leq C h^2 \int_{\hat{Q}_{h}(a')} \dist^2(\nabla_h y_h, \SO(3)) \, \mathrm dx.
\end{flalign}
We now define the piecewise linear function $r_{h}: \omega \rightarrow \R^3$ so that
\begin{equation*}
  r_{h}(x') 
  = \begin{cases} 
	     r_{h,a}(x',0) & \text{if } x' \in a' + (0,h)^2 \text{ with } a' \in \mathcal{G}_{h}^{\rm el}, \\ 
       (x',0) & \textrm{otherwise}  
    \end{cases}
\end{equation*} 
and notice that 
\begin{equation}\label{eq:r-planar-nonplanar}
  |r_{h,a'}(x) - r_{h,a'}(x',0)|
  = |R_{h,a'}(0,0,hx_3)| 
	\leq h. 
\end{equation} 
From \eqref{estimate on one cube II}, \eqref{eq:r-planar-nonplanar} and the energy estimate we get
\begin{flalign*}
\begin{split}
	&\int_{\Omega} \chi_{h}^{\rm el}(x) |y_h(x) - r_{h}(x')|^{2} \, \mathrm dx \\ 
	&\quad \leq C h^2 \int_{\Omega} \chi_{h}^{\rm el} \big( 1 + \dist^2(\nabla_h y_h, \SO(3)) \big) \, \mathrm dx 
	\leq C h^{2}
\end{split}
\end{flalign*} 
and, since $ \|y_h\|_{L^\infty(\Omega)} \leq C $, we may use \eqref{number of bad cubes} to conclude 
\begin{equation}\label{L2 distance of the rescaled gradients from piecewice rotational function:2'}
  \lim_{h \to 0} \|\chi_{h}^{\rm el} ( y_h - r_{h} ) \|_{L^2(\Omega)}=0.
\end{equation}

\medskip 

\noindent {\bf Estimates on two adjacent good cubes.} If $a', b' \in h \Z^{2}$ and $|a'-b'|=h$ we define
\begin{equation*}
  Q_{h}(a',b') 
  = \big( \overline{Q_{h}(a')} \cup \overline{Q_{h}(a')} \big)^{\circ}, \qquad  
  \hat{Q}_{h}(a',b') 
  = \big( \overline{\hat{Q}_{h}(a')} \cup \overline{\hat{Q}_{h}(a')} \big)^{\circ}. 
\end{equation*}
We fix $a', b' \in \mathcal{G}_{h}^{\rm el}$ with $|a'-b'|=h$. We apply again \cite[Theorem 3.1]{FrieseckeJamesMueller:02} to select a constant $C$ and a rotation $R_{h,a,b}$ such that, after rescaling, we have
\begin{flalign}\label{estimate on two cubes}
  \int_{\hat{Q}_{h}(a',b')} |\nabla_h y_h - R_{h,a,b}|^{2} \, \mathrm dx 
  \leq C \int_{\hat{Q}_{h}(a',b')} \dist^2(\nabla_h y_h, \SO(3)) \, \mathrm dx. 
\end{flalign}
Therefore combining \eqref{estimate on one cube I} and \eqref{estimate on two cubes} we get
\begin{flalign}\label{estimate on two cubes:1}
  |R_{h,a'} - R_{h,a',b'}|^{2} 
  &\leq 2 \dashint_{\hat{Q}_{h}(a')} |\nabla_h y_h - R_{h,a'}|^{2} \, \mathrm dx \nonumber \\
  &\quad + 2 \dashint_{\hat{Q}_{h}(a')} |\nabla_h y_h - R_{h,a',b'}|^{2} \, \mathrm dx \notag\\
  &\leq C \dashint_{\hat{Q}_{h}(a',b')} \dist^2(\nabla_h y_h, \SO(3)) \, \mathrm dx.
\end{flalign}
The same inequality is valid with $a'$ and $b'$ interchanged, whence we deduce
\begin{equation}\label{estimate for adjacent rotations}
  |R_{h,a'} - R_{h,b'}|^{2} 
   \leq C \dashint_{\hat{Q}_{h}(a',b')} \dist^2(\nabla_h y_h, \SO(3)) \, \mathrm dx.
\end{equation}

We define 
\begin{equation*}
  c_{h,a',b'} 
  = \dashint_{Q_h(a',b')} ( v_h(x) - R_{h,a',b'} x ) \, \mathrm dx 
\end{equation*}
and the linear map $r_{h,a',b'} : \R^{3} \rightarrow \R^{3} $ by
\begin{equation*}
  r_{h,a',b'}(x) = R_{h,a',b'}(x', h x_3) + c_{h,a',b'}.
\end{equation*}
We use now the Poincarè-Wirtinger inequality on the domain $Q_h(a',b')$ and \eqref{estimate on two cubes} to estimate
\begin{flalign*}
  \int_{\hat{Q}_{h}(a',b')} |y_h - r_{h,a',b'}|^{2} \, \mathrm dx 
  &= h^{-1} \int_{Q_{h}(a',b')} |v_h(x) - r_{h,a',b'}(x',hx_3)|^{2} \, \mathrm dx \\
  &\leq C h \int_{Q_{h}(a',b')} |\nabla v_h - R_{h,a',b'}|^{2} \, \mathrm dx \\
  &\leq C h^2 \int_{\hat{Q}_{h}(a',b')} \dist^2(\nabla_h y_h, \SO(3)) \, \mathrm dx. 
\end{flalign*}
Combining this inequality with \eqref{estimate on one cube II} we get
\begin{equation}\label{estimates on two cubes 1}
  \int_{\hat{Q}_{h}(a')} |r_{h,a',b'} - r_{h,a'}|^{2} \, \mathrm dx 
  \leq C h^2 \int_{\hat{Q}_{h}(a',b')} \dist^2(\nabla_h y_h, \SO(3)) \, \mathrm dx. 
\end{equation}

Now we notice that the constants satisfy the relation 
\begin{flalign*}
& c_{h,a'} + c_{h,b'} - 2 c_{h,a',b'} \\
& \qquad  =  \dashint_{Q_h(a')} (R_{h,a',b'} - R_{h,a'}) x \, \mathrm dx 
            +\dashint_{Q_h(b')} (R_{h,a',b'} - R_{h,b'}) x \, \mathrm dx;
\end{flalign*}
then employing Jensen's inequality and \eqref{estimate on two cubes:1} we estimate
\begin{flalign*}
  | c_{h,a'} + c_{h,b'} - 2 c_{h,a',b'} |^{2} 
  \leq C \dashint_{\hat{Q}_{h}(a',b')} \dist^2(\nabla_h y_h, \SO(3)) \, \mathrm dx. 
\end{flalign*}
By \eqref{estimate on two cubes:1} we then infer
\begin{equation*}
  |r_{h,b'}(x) + r_{h,a'}(x) - 2r_{h,a',b'}(x)|^{2} 
	\leq C \dashint_{\hat{Q}_{h}(a',b')} \dist^2(\nabla_h y_h, \SO(3)) \, \mathrm dx 
\end{equation*}
whenever $x \in \Omega$; hence we conclude with the help of \eqref{estimates on two cubes 1}
\begin{flalign*}
  &\int_{\hat{Q}_{h}(a')} |r_{h,b'} - r_{h,a'}|^{2} \, \mathrm dx \\ 
  &\quad  \leq 2 \int_{\hat{Q}_{h}(a')} |r_{h,b'} + r_{h,a'} - 2r_{h,a',b'}|^{2} \, \mathrm dx 
    + 2 \int_{\hat{Q}_{h}(a')} |2r_{h,a',b'} - 2r_{h,a'}|^{2} \mathrm dx \notag\\
  &\quad  \leq C h^{2} \dashint_{\hat{Q}_{h}(a',b')} \dist^2(\nabla_h y_h, \SO(3)) \, \mathrm dx. 
\end{flalign*}
If $c' \in h \Z^2$ a change of variables and \eqref{estimate for adjacent rotations} give
\begin{flalign*}
&\int_{\hat{Q}_{h}(c')} |r_{h,b'} - r_{h,a'}|^{2} \, dx \\
& \quad \leq 2 \int_{\hat{Q}_{h}(a')} |r_{h,b'} - r_{h,a'}|^{2} \, \mathrm dx 
  + 2 \int_{\hat{Q}_{h}(a')}|(R_{h,b'}' - R_{h,a'}')(a' - c')|^{2} \mathrm dx \\
& \quad \leq C \big( h^2 + |a' - c'|^2 \big) \dashint_{\hat{Q}_{h}(a',b')} \dist^2(\nabla_h y_h, \SO(3)) \, \mathrm dx.
\end{flalign*}
Together with \eqref{eq:r-planar-nonplanar} this yields 
\begin{flalign}\label{estimate for adjacent rotations:1}
&\int_{c'+(0,h)^2}	|r_{h,b'}(x',0) - r_{h,a'}(x',0)|^{2} \, \mathrm dx' \notag \\
& \quad \leq C \big( h^2 + |a' - c'|^2 \big) \dashint_{\hat{Q}_{h}(a',b')} \dist^2(\nabla_h y_h, \SO(3)) \, \mathrm dx + C h^4.
\end{flalign}
\medskip

\noindent {\bf Interpolation.} For $a' \in h \Z^2$ and $t > 0$ we define the parallel squares $Q'_{h,t}(a') = a' + (\frac{h-t}{2}, \frac{h+t}{2})$. 
We fix $\eta = \frac{1}{5} h$, we define
\begin{equation*}
  \omega_{h} 
	= \omega \setminus \bigcup_{a' \in \mathcal{G}_h^{\rm v} \cup \mathcal{B}_{h}} \overline{Q'_{h, h + \eta}(a')}. 
\end{equation*}
For later use we remark that, similarly as in \eqref{eq:Dhprime-Dh}, \eqref{eq:vol-sh-tube} and \eqref{number of bad cubes} imply that $\Omega \setminus (\omega_h \times (0,1)) \supset D_h' \supset D_h$ satisfies 
\begin{align}\label{eq:omegah-Dh} 
   \mathcal{L}^3 \big( \Omega \setminus (\omega_h \times (0,1)) \setminus D_h \big) \le C h. 
\end{align}

Let $\{ \psi_{h,a'}: a' \in h \Z^2 \}$ be a family of smooth functions such that $0 \leq \psi_{h,a'} \leq 1$,
\begin{equation*}
  \spt \psi_{h,a'} \subset  Q'_{h,h + \eta}(a'), \quad 
	\psi_{h,a'} \equiv 1 \text{ on } Q'_{h,h - \eta}(a'), \quad  
	\sum_{a' \in h \Z^2} \psi_{h,a'} \equiv 1 \text{ on } \omega 
\end{equation*}
and, moreover, $\|\nabla \psi_{h,a'} \|_{\infty} \leq 10 h^{-1}$. We define $\tilde{R}_{h} : \omega \rightarrow \R^{3 \times 3}$ and $\tilde{r}_{h}: \omega \rightarrow \R^{3}$ such that 
\begin{align*}
  \tilde{R}_{h}(x') 
  &= \sum_{a' \in \mathcal{G}_{h}^{\rm el}}\psi_{h,a'}(x') R_{h,a'} \; \; \text{if } x' \in \omega_{h}, 
  & \tilde{R}_{h}(x') 
  &= \Id \; \; \text{if } x' \in \omega \setminus \omega_{h}, \\ 
  \tilde{r}_{h}(x') 
  &= \sum_{a' \in \mathcal{G}_{h}^{\rm el}} \psi_{h,a'}(x') r_{h,a'}(x',0) \; \; \textrm{if } x' \in \omega_{h}, 
  & \tilde{r}_{h}(x') 
  &= (x',0) \; \; \textrm{if } x' \in \omega \setminus \omega_{h}.
\end{align*}
Notice $\tilde{R}_{h} \in SBV(\omega, \R^{3 \times 3})$ and $\tilde{r}_{h} \in SBV(\omega, \R^{3})$ with 
\begin{equation}\label{eq:infty-bd-rR}
  \| \tilde{R}_{h} \|_{\infty} \leq C \quad\text{and}\quad 
  \| \tilde{r}_{h} \|_{\infty} \leq C
\end{equation}
for some $C > 0$ and all $h > 0$. Since $ J_{\tilde{r}_{h}}$ and $J_{\tilde{R}_{h}}$ are contained in $\omega \cap \partial \omega_h$ which is covered by $\bigcup_{a' \in \mathcal{B}_{h}} \partial Q'_{h,h + \eta}(a')$, we use \eqref{number of bad cubes} to conclude
\begin{equation}\label{jump set of the approximating rotations}
  \mathcal{H}^{1}(J_{\tilde{r}_{h}} \cup J_{\tilde{R}_{h}}) 
	\leq Ch \# \mathcal{B}_{h} 
	\leq C.
\end{equation}
Let $a' \in h \Z^2$. For $\mathcal{L}^2$-a.e.\ $x' \in Q'_{h}(a')$ we remark that there are three possibilities for the number of (enlarged) squares containing $x'$: $\# \{b' \in h \Z^2 : Q'_{h,h + \eta}(b') \ni x \} \in \{1, 2, 4\}$. For $x' \in Q'_{h,h - \eta}(a')$ this number is $1$. Near the corners, on 
$$ \sigma^c(a') 
   := \bigcup_{b' \in \{a \pm h e_1 \pm h e_2\}} Q'_{h}(a') \cap Q'_{h,h + \eta}(b'), $$ 
it is $4$. Finally, near the lateral boundary but away from the corners, on 
$$ \sigma^l(a') 
   := Q'_{h}(a') \setminus \overline{Q'_{h,h - \eta}(a') \cup \sigma^c(a')} $$ 
it is $2$. 
\medskip 

\noindent {\bf Estimates on the interpolation error.} Next we prove that
\begin{equation}\label{same convergence of piecewise functions and partition of unity}
  \int_{\omega_{h}}|\tilde{R}_{h}-R_{h}|^{2} \, \mathrm d x' 
	\leq Ch^{2}, \quad \int_{\omega_{h}}|\tilde{r}_{h} - r_{h}|^{2} \, \mathrm d x'  
	\leq C h^{2}.
\end{equation}
The proof is respectively based on the key estimates \eqref{estimate for adjacent rotations} and \eqref{estimate for adjacent rotations:1}. We prove in detail only the second inequality, since the first can be proved along the same lines. 

Suppose $a' \in \mathcal{G}_{h}^{\rm el}$. We consider the three parts of $Q'_{h}(a')$ separately. 
\smallskip 

\noindent{\em 1:} On $Q'_{h, h - \eta}(a')$ one has  $\tilde{r}_{h} = r_{h}$. 
\smallskip  

\noindent{\em 2:} On $\sigma^l(a')$ we use \eqref{estimate for adjacent rotations:1} with $c' = a'$ in order to estimate 
\begin{flalign*}
  \int_{\sigma^l(a') \cap \omega_{h}} |\tilde{r}_{h} - r_{h}|^{2} \, \mathrm dx 
  &\leq \sum_{\substack{b' \in \mathcal{G}_{h}^{\rm el} \\ |b' - a'| = h}} \int_{Q'_{h,h}(a') \cap Q'_{h, h + \eta}(b')} |r_{h,b'} - r_{h,a'}|^{2} \, \mathrm dx \\
  &\leq C h^{2} \sum_{\substack{b' \in \mathcal{G}_{h}^{\rm el} \\ |b' - a'| = h}} \dashint_{\hat{Q}_{h}(a',b')} \dist^2(\nabla_h y_h, \SO(3)) \, \mathrm dx + Ch^4. 
\end{flalign*}
\smallskip  

\noindent{\em 3:} On $\sigma^c(a')$ we first note that if $\sigma^c_{s_1s_2}(a') := Q'_{h, h}(a') \cap Q'_{h, h + \eta}(b') \subset \omega_{h}$ for $b' = a' + s_1 e_1 + s_2 e_2$ with $s_1, s_2 \in \{-h, +h\}$, then $b'(i_1,i_2) = a' + i_1 s_1 e_1 + i_2 s_2 e_2 \in \mathcal{G}_{h}^{\rm el}$ for all $0 \le i_1, i_2 \le 1$ and 
\begin{flalign*}
  \int_{\sigma^c_{s_1s_2}(a')} |\tilde{r}_{h} - r_{h}|^{2} \, \mathrm dx 
  &\leq \int_{\sigma^c_{s_1s_2}(a')} \sum_{0 \le i_1, i_2 \le 1} |r_{h,b'(i_1,i_2)} - r_{h,a'}|^{2} \, \mathrm dx. 
\end{flalign*}
The summands corresponding to $(i_1,i_2) \ne (1,1)$ can be directly estimated using \eqref{estimate for adjacent rotations:1} with $c' = a'$; instead, for the summand corresponding to $(i_1,i_2) = (1,1)$ we first notice
\begin{flalign}\label{eq:diagonal-estimate}
\begin{split}
  &\int_{Q'_{h, h}(a')} |r_{h, b'(1,1)} - r_{h, a'} |^{2} \mathrm dx \\
  &\quad \leq 2 \int_{Q'_{h, h}(a')}|r_{h, b'(1,1)} - r_{h, b'(1,0)}|^{2} \, \mathrm dx 
    +  2\int_{Q'_{h, h}(a')}|r_{h, b'(1,0)} - r_{h, a'}|^{2} \, \mathrm dx,
\end{split}
\end{flalign}
and then apply \eqref{estimate for adjacent rotations:1}; hence 
\begin{equation*}
  \int_{\sigma^c_{s_1s_2}(a')} |\tilde{r}_{h} - r_{h}|^{2} \, \mathrm dx 
  \leq C h^{2} \sum_{0 \le i_1 + i_2 \le 1} \dashint_{\hat{Q}_{h}(b'(i_1,i_2))} \dist^2(\nabla_h y_h, \SO(3)) \, \mathrm dx + C h^4.
\end{equation*}
Summing over $b'$ we obtain 
\begin{flalign*}
  \int_{\sigma^c(a') \cap \omega_{h}} |\tilde{r}_{h} - r_{h}|^{2} \, \mathrm dx 
  & \leq C h^{2} \sum_{\substack{b' \in \mathcal{G}_{h}^{\rm el}, \\ |b'-a'| \le \sqrt{2} h}} \dashint_{\hat{Q}_{h}(b')} \dist^2(\nabla_h y_h, \SO(3)) \, \mathrm dx + C h^4. 
\end{flalign*}
\smallskip

Combining these estimates and summing over $a' \in \mathcal{G}_{h}^{\rm el}$, we get the second inequality in \eqref{same convergence of piecewise functions and partition of unity} with the help of the energy estimate.

Since $\| r_{h,a'}\|_{L^{\infty}(\omega)} \leq C $ for all $a' \in \mathcal{G}_{h}^{\rm el}$, it follows that $ \| \tilde{r}_{h} - r_{h} \|_{\infty} \leq C $; moreover $ \| \tilde{R}_{h} - R_{h}\|_{\infty} \leq C$. Since $\tilde{r}_{h}(x') = (x',0) = r_{h}(x')$ and $\tilde{R}_{h}(x') = \Id = R_{h}(x')$ for $x \in Q'_{h, h + \eta}(a')$, $a' \in \mathcal{G}_{h}^{\rm v}$, it follows from \eqref{same convergence of piecewise functions and partition of unity} and \eqref{number of bad cubes} that
\begin{equation}\label{same convergence of piecewise functions and partition of unity:2}
  \lim_{h \to 0} \|\tilde{R}_{h} - R_{h}\|_{L^{2}(\omega)} = 0, \quad 
  \lim_{h \to 0} \|\tilde{r}_{h} - r_{h}\|_{L^{2}(\omega)}=0.
\end{equation}

\medskip 

\noindent {\bf Gradient estimates for the interpolation.} 
We prove now 
\begin{equation}\label{L2 estimate for the abs cont part of the gradient}
  \int_{\omega} |\nabla \tilde{R}_{h}|^{2} \, dx \leq C, \qquad 
  \int_{\omega} |\nabla \tilde{r}_{h}|^{2} \, dx \leq C.
\end{equation}
As for \eqref{same convergence of piecewise functions and partition of unity} the proof of these inequalities is respectively based on \eqref{estimate for adjacent rotations} and \eqref{estimate for adjacent rotations:1} and we provide details only for the second. We notice that
\begin{equation*}
  \nabla \tilde{r}_{h} 
  = \chi_{\omega_{h}} \big( G_{h} + \tilde{R}'_{h} \big) 
  + \chi_{\omega \setminus \omega_{h}} \Id' 
\end{equation*}
where $G_{h} =  \sum_{a' \in \mathcal{G}_{h}^{\rm el}} r_{h,a'} \otimes \nabla \psi_{h,a'}$. Since $\| \tilde{R}'_{h}\|_{\infty} \leq C$, it remains to prove that $\|G_{h}\|_{L^{2}(\omega_{h})} \leq C$. To this end we again estimate on the three parts of a square $Q'_{h}(a')$, $a' \in \mathcal{G}_{h}^{\rm el}$, separately. 
\smallskip 

\noindent{\em 1:} On $Q'_{h, h - \eta}(a')$ we have $G_{h} = r_{h,a'} \otimes \nabla \psi_{h,a'} = 0$. 
\smallskip 

\noindent{\em 2:} For an estimate on $\sigma^l(a')$ we first observe that for $b' \in \{\pm h e_1, \pm h e_2 \}$ on the set $\sigma^l(a') \cap Q_{h, h + \eta}'(b') \cap \omega_{h}$ (which is non-empty only if $b' \in \mathcal{G}_{h}^{\rm el}$) we have $\nabla \psi_{h,a'} + \nabla \psi_{h,b'} = 0$. Using \eqref{estimate for adjacent rotations:1} we estimate
\begin{flalign*}
  \int_{\sigma^l(a') \cap \omega_{h}} |G_{h}|^{2} \, \mathrm dx 
  &= \sum_{\substack{b' \in \mathcal{G}_{h}^{\rm el} \\ |b' - a'| = h}} \int_{Q'_{h, h}(a') \cap Q'_{h, h + \eta}(b')} |r_{h,a'} \otimes \nabla \psi_{h,a'} + r_{h,b'} \otimes \nabla \psi_{h,b'}|^{2} \, \mathrm dx \\ 
  &= \sum_{\substack{b' \in \mathcal{G}_{h}^{\rm el} \\ |b' - a'| = h}} \int_{Q'_{h, h}(a') \cap Q'_{h, h + \eta}(b')} |( r_{h,a'} - r_{h,b'} ) \otimes \nabla \psi_{h,a'}|^{2} \, \mathrm dx \\
  &\leq C h^{-2} \sum_{\substack{b' \in \mathcal{G}_{h}^{\rm el} \\ |b' - a'| = h}} \int_{Q'_{h, h}(a') \cap Q'_{h, h + \eta}(b')} |r_{h,a'} - r_{h,b'}|^{2} \, \mathrm dx \\ 
  &\leq C \sum_{\substack{b' \in \mathcal{G}_{h}^{\rm el} \\ |b' - a'| = h}} \dashint_{\hat{Q}_{h}(a',b')} \dist^2(\nabla_h y_h, \SO(3)) \, \mathrm dx + C h^2. 
\end{flalign*}
\smallskip  

\noindent{\em 3:} On $\sigma^c(a')$ we can argue similarly. If $\sigma^c_{s_1s_2}(a') := Q'_{h, h}(a') \cap Q'_{h, h + \eta}(b') \subset \omega_{h}$ for $b' = a' + s_1 e_1 + s_2 e_2$ with $s_1, s_2 \in \{-h, +h\}$, then with $b'(i_1,i_2) = a' + i_1 s_1 e_1 + i_2 s_2 e_2$ 
$$ G_{h} 
   = \sum_{0 \le i_1, i_2 \le 1} r_{h, b'(i_1,i_2)} 
     \otimes \nabla \psi_{h,b'(i_1,i_2)} $$ 
on $\sigma^c_{s_1s_2}(a')$ and using $\sum_{0 \le i_1, i_2 \le 1} \nabla \psi_{h,b'(i_1,i_2)} = 0$ on this set we get 
\begin{flalign*}
  \int_{\sigma^c_{s_1s_2}(a')} |G_{h}|^{2} \, \mathrm dx 
  &\leq 4 \int_{\sigma^c_{s_1s_2}(a')} \sum_{0 \le i_1, i_2 \le 1} |( r_{h, b'(i_1,i_2)} - r_{h,a'}) \otimes \nabla \psi_{h, b'(i_1,i_2)}|^{2} \, \mathrm dx \\ 
  &\leq C h^{-2} \int_{\sigma^c_{s_1s_2}(a')} \sum_{0 \le i_1, i_2 \le 1} |r_{h, b'(i_1,i_2)} - r_{h,a'}|^{2} \, \mathrm dx \\ 
  &\leq C \sum_{0 \le i_1 + i_2 \le 1} \dashint_{\hat{Q}_{h}(b'(i_1,i_2))} \dist^2(\nabla_h y_h, \SO(3)) \, \mathrm dx + C h^2 
\end{flalign*}
with the help of \eqref{eq:diagonal-estimate} and \eqref{estimate for adjacent rotations:1} as above. This proves 
\begin{flalign*}
  \int_{\sigma^c(a') \cap \omega_{h}} |G_{h}|^{2} \, dx 
  &\leq C \sum_{\substack{b' \in \mathcal{G}_{h}^{\rm el}, \\ |b'-a'| \le \sqrt{2} h}} \dashint_{\hat{Q}_{h}(b')} \dist^2(\nabla_h y_h, \SO(3)) \, \mathrm dx + C h^2. 
\end{flalign*} 
Summing over good squares we find that $\|G_{h}\|_{L^{2}(\omega_{h})} \leq C$. 
\medskip 

\noindent {\bf Convergence.} 
In view of \eqref{eq:infty-bd-rR}, \eqref{jump set of the approximating rotations} and \eqref{L2 estimate for the abs cont part of the gradient} a standard $SBV$ compactness result (cf.\ \cite[Theorems 4.7 and 4.8]{AmbrosioFuscoPallara:00}) implies the existence of two maps $r \in L^{\infty}(\omega, \R^{3}) \cap SBV(\omega, \R^{3})$ and $R \in L^{\infty}(\omega, \R^{3 \times 3}) \cap SBV(\omega,\R^{3 \times 3})$ with $\mathcal{H}^{1}(J_{r}) < \infty$, $\mathcal{H}^{1}(J_{R}) < \infty$ such that, up to subsequences, 
\begin{align}
  \tilde{r}_{h} &\to r \quad \text{in } L^{2}(\omega, \R^{3}), & 
  \nabla' \tilde{r}_{h} &\rightharpoonup \nabla' r \quad \text{in } L^{2}(\omega,\R^{3 \times 2}), \label{weak convergente of approx. gradients-r} \\ 
  \tilde{R}_{h} &\to R \quad \text{in } L^{2}(\omega,\R^{3 \times 3}), & 
  \partial_{i} \tilde{R}_{h} &\rightharpoonup \partial_{i} R \quad \text{in } L^{2}(\omega,\R^{3 \times 3}) \text{ for } i = 1,2. \label{weak convergente of approx. gradients-R}
\end{align}

We observe that $R(x) \in \SO(3)$ for $\mathcal{L}^{2}$ a.e.\ $x \in \omega$. Moreover we can combine \eqref{number of bad cubes}, \eqref{L2 distance of the rescaled gradients from piecewice rotational function:2}, \eqref{L2 distance of the rescaled gradients from piecewice rotational function:2'}, \eqref{same convergence of piecewise functions and partition of unity:2}, \eqref{weak convergente of approx. gradients-r} and \eqref{weak convergente of approx. gradients-R} to conclude that, up to subsequences, 
\begin{equation}\label{same convergence of piecewise functions and partition of unity:1}
  \lim_{h \to 0}	\|\chi_{h}^{\rm el} (y_h - r)\|_{L^{2}(\Omega)} = 0, \quad 
  \lim_{h \to 0} \|\chi_{h}^{\rm el} (\nabla_{h} y_h - R)\|_{L^{2}(\Omega)} = 0.
\end{equation}
Together with \eqref{Compactness:2} this proves \eqref{Compactness:3} and \eqref{Compactness:4}, as soon as we have shown that 
\begin{equation}\label{eq:R-nabla-r}
  R = (\nabla' r, \partial_1 r \wedge \partial_2 r )\text{ on }\omega \setminus D. 
\end{equation}

To see this we notice that $\| y_h \|_{\infty} \leq C$, $\mathcal{H}^{2}(J_{\chi_{h}^{\rm el} y_h}) \leq C$ by \eqref{number of bad cubes}, $\nabla (\chi_{h}^{\rm el} y_h)= \chi_{h}^{\rm el} \nabla y_h$ and $\|\nabla (\chi_{h}^{\rm el} y_h)\|_{L^{2}(\Omega)} \leq C$ by \eqref{same convergence of piecewise functions and partition of unity:1}. Then we apply the $SBV$ compactness theorem to the sequence $ \chi_{h}^{\rm el} y_h$ and we infer from \eqref{Compactness:3} that, up to subsequences,
\begin{equation*}
  \chi_{h}^{\rm el} y_h \to \chi_{\omega \setminus D} r \quad \text{in } L^{2}(\Omega, \R^{3}), \qquad 
  \chi_{h}^{\rm el} \nabla y_h \rightharpoonup \chi_{\omega \setminus D}  (\nabla' r, 0) \quad \text{in } L^{1}(\Omega, \R^{3 \times 3}). 
\end{equation*}
Hence, $\chi_{\omega \setminus D} \nabla' r = \chi_{\omega \setminus D} R'$ by \eqref{Compactness:4}. Since $R \in \SO(3)$ a.e.\ this concludes the proof.  
\end{proof}

\section{The lower bound}

In this section we prove Theorem~\ref{theo:main}(i). Let $(y_h, D_h)$ be a sequence and $(y, D)$ an element in $L^\infty(\Omega, \mathbb{R}^3) \times \mathcal{F}(\Omega)$ such that $y_h \to y$ in $L^1(\Omega, \mathbb{R}^3)$, $\| y_h \|_{L^{\infty}} \le C$, $\chi_{D_h} \to \chi_D$ in $L^1(\Omega, \mathbb{R})$ and $ (D_h)_{0 < h < 1} $ satisfies the $ \psi $-minimal droplet assumption. Without loss of generality we pass to subsequences (not relabeled) in the following and assume that $\lim_{h \to 0} \mathcal{E}_h(y_h, D_h)$ exists and is finite. In particular, $(y_h, D_h) \in W^{1,2}(\Omega, \mathbb{R}^3) \times \mathcal{F}(\Omega)$ for all $h$ and $(y, D) \in SBV^{2,2}_{\rm iso}(\omega) \times \mathcal{F}(\omega)$ by Theorem~\ref{Compactness}. We will moreover make use of the auxiliary functions defined and estimates obtained in the proof of Theorem~\ref{Compactness}. 

We will provide estimates from below on the elastic part and the surface part separately.  
\medskip 

\noindent {\bf Lower bound for the surface part.} 
We fix $s < 1$ and consider the layers $\Omega_j = \omega \times (\frac{j-1}{n}, \frac{j}{n})$ of height $\frac{h}{n}$, where $n \in \mathbb{N}$ is such that $s^2 n > 1+\max\{ \psi(y) : y \in [-1,1]^3 \}$. On each such layer we proceed exactly as in the previous section (with $\Omega$ replaced by $\Omega_j$, $h$ by $\frac{h}{n}$ and $(y_h, D_h)$ by $(y_h|_{\Omega_j}, D_h \cap \Omega_j)$ to define interpolations $\tilde{r}_{j,h}$ and $\tilde{R}_{j,h}$ of rigid motions and rotations, respectively, on cubes of side-length $\frac{h}{n}$. By \eqref{weak convergente of approx. gradients-r}, \eqref{weak convergente of approx. gradients-R}, \eqref{eq:R-nabla-r} we have (passing to subsequences) 
\begin{align}
  \tilde{r}_{j,h} &\to r_{j,0} \quad \text{in } L^{2}(\omega, \R^{3}), & 
  \nabla' \tilde{r}_{j,h} &\rightharpoonup \nabla' r_{j,0} \quad \text{in } L^{2}(\omega,\R^{3 \times 2}), \label{eq:tr-conv} \\ 
  \tilde{R}_{j,h} &\to R_{j,0} \quad \text{in } L^{2}(\omega,\R^{3 \times 3}), & 
  \partial_{i} \tilde{R}_{j,h} &\rightharpoonup \partial_{i} R_{j,0} \quad \text{in } L^{2}(\omega,\R^{3 \times 3}),\ i = 1,2, \label{eq:tR-conv}
\end{align}
for suitable $(r_{j,0}, D_{j,0}) \in SBV^{2,2}_{\rm iso}(\omega) \times \mathcal{F}(\omega)$ and $R_{j,0} = (\nabla' r_{j,0}, \partial_1 r_{j,0} \wedge \partial_2 r_{j,0} )$. (Note that $R_{j,0} = \Id$ and $r_{j,0} = \id$ on $D_{j,0}$ by construction.) 

Observe that the set $A_h'$ for the whole plate $\Omega_h$ in  particular contains the little cubes of sidelength $\frac{h}{n}$ within the $j$-th layer that intersect $\partial \Omega$ or whose closure intersects the closure of $A_h^{(\sigma_{j,h/n})}$ since $\sigma_{j,h/n} \le s \frac{h}{n} \le s^2 h \le \sigma_h$. Thus $D_h' \cap \Omega_j \supset D_{j,h}'$. As moreover $\mathcal{L}^3 \big( (D_h' \cap \Omega_j) \setminus D_{j,h}' \big) \le \mathcal{L}^3 \big( (D_h' \setminus D_h) \cap \Omega_j \big) \to 0$, we get from \eqref{Compactness:3} in Theorem~\ref{Compactness} for $\Omega_j$ 
$$ \chi_{\Omega_j \setminus D_h'} y_h|_{\Omega_j} \to \chi_ {\omega \setminus D_{j,0}} r_{j,0}|_{\Omega_j}, $$
while \eqref{Compactness:3} in Theorem~\ref{Compactness} for $\Omega$ gives 
$$ \chi_{\Omega_j \setminus D_h'} y_h|_{\Omega_j} \to \chi_ {\omega \setminus D} r|_{\Omega_j}. $$
It follows that $r_{j,0} = r$, $R_{j,0} = R$ and $D_{j,0} = D$ for each $j$.

Our choice of $n$ guarantees that $(-\tfrac{h}{n}, \tfrac{h}{n})^3 \subset \{ y \in \mathbb{R}^3 : \psi^{\circ}(y) \le s^2 h \}$ and hence that any interior bad cuboid of the layer $\Omega_j$ is contained in $D_h^{(2 \sigma_h)_h} \cap \Omega_j$. Let $\theta_{h} : \Omega \to [0,1]$ be a smooth cut-off function such that $\theta_{h} \equiv 0$ on $D_h^{(2 \sigma_h)_h}$ and $\theta_{h} \equiv 1$ on $\Omega \setminus D_h^{(3 \sigma_h)_h}$. We recall from \eqref{eq:vol-sh-tube} that 
\begin{align}\label{eq:vol-3sh}
  \mathcal{L}^3(D_h^{(3 \sigma_h)_h} \Delta D_h) \le Ch.  
\end{align}

Let $\omega' \subset \subset \omega$ have a Lipschitz boundary and set $\Omega'_j = \omega' \times (\frac{j-1}{n}, \frac{j}{n})$. We define $f_{j, h} \in W^{1,2}(\Omega'_j, \R^3 \times \R^{3 \times 3})$ by 
$$ f_{j, h}(x) 
   = \theta_{h}(x)\big(\tilde{r}_{j,h}(x), \tilde{R}_{j,h}(x) \big) 
     + \big( 1 - \theta_{h}(x) \big) (x, \Id). $$ 
Recalling \eqref{eq:infty-bd-rR} we see that $\| f_{j, h} \|_{\infty} \le C$. Moreover, by \eqref{L2 estimate for the abs cont part of the gradient} we also have 
\begin{align}\label{eq: L2-bd-fjh}
  \| \nabla f_{j, h} \|_{L^2(\Omega_j' \setminus D_h^{(3 \sigma_h)_h})} \le C. 
\end{align}

Now consider the functionals $\mathcal{K}_\eps$ on $W^{1,2}(\Omega'_j, \R^3 \times \R^{3 \times 3}) \times \mathcal{F}(\Omega'_j)$ given as 
$$ \mathcal{K}_\eps(f, E) 
   = \eps \int_{\Omega'_j \setminus E} |\nabla f(x)|^2 \, \mathrm dx 
     + \int_{\partial^* E \cap \Omega'_j} \psi_0(\nu'(E)) + \eps |\nu_3(E)|\, \mathrm d \mathcal{H}^2. $$
The $L^2$ bound in \eqref{eq: L2-bd-fjh} shows that  
\begin{align*}
  \int_{\partial D_h^{(3 \sigma_h)_h} \cap \Omega'_j} \psi(\nu_h(D_h^{(3 \sigma_h)_h})) \, \mathrm d \mathcal{H}^2 
  &\ge \int_{\partial D_h^{(3 \sigma_h)_h} \cap \Omega'_j} \psi_0(\nu'(D_h^{(3 \sigma_h)_h})) \, \mathrm d \mathcal{H}^2 \\ 
  &\ge \mathcal{K}_\eps(f_{j, h},  D_h^{(3 \sigma_h)_h}) - C \eps. 
\end{align*} 
By \eqref{eq:tr-conv}, \eqref{eq:tR-conv} and \eqref{eq:vol-3sh} we have $f_{j,h} \to (r, R)$ in $L^1(\Omega'_j, \R^3 \times \R^{3 \times 3})$ and $\chi_{D_h^{(3 \sigma_h)_h}} \to \chi_{D}$ in $L^1(\Omega'_j)$. We may thus invoke the general Relaxation Theorem~\ref{theo:bulk-relax} to deduce 
\begin{align*}
  &\liminf_{h \to 0} \mathcal{K}_\eps(f_{j, h},  D_h^{(3 \sigma_h)_h})
  \ge \mathcal{K}_\eps^{\rm rel} \big( (r, R), D \big) \\ 
  &= \eps \int_{\Omega'_j \setminus D} |(\nabla (r, R)(x)|^2 \, \mathrm dx 
  + 2 \int_{\Omega'_j \cap J_{(r,R)} \cap D^0} \psi_0 (\nu'(r, R)) + \eps |\nu_3(r, R)| \, \mathrm d\mathcal{H}^2 \\
  &\quad + \int_{\Omega'_j \cap \partial^* D} \psi_0 (\nu'(D)) + \eps |\nu_3(D)| \, \mathrm d\mathcal{H}^2. 
\end{align*} 
It follows that 
\begin{align*}
  &\liminf_{h \to 0} \int_{\partial^* D_h^{(3 \sigma_h)_h} \cap \Omega'_j} \psi(\nu_h(D_h^{(3 \sigma_h)_h})) \, \mathrm d \mathcal{H}^2 \\ 
  &\quad \ge 2 \int_{\Omega'_j \cap J_{(r,R)} \cap D^0} \psi_0 (\nu'(r, R)) \, \mathrm d\mathcal{H}^2 
  + \int_{\Omega'_j \cap \partial^* D} \psi_0 (\nu'(D)) \, \mathrm d\mathcal{H}^2 - C \eps. 
\end{align*} 
We now sum over $j$ and make use of our specific choice of $\sigma_h$ in \eqref{eq:A3sh-surf-est}. 
Noting that $\mathcal{H}^2(\Omega \cap J_{(r,R)} \cap \partial \Omega_j) = \mathcal{H}^2(\Omega \cap \partial^* D \cap \partial \Omega_j) = 0$ we arrive at 
\begin{align*}
  &\liminf_{h \to 0} \int_{\partial^* D_h \cap \Omega} \psi(\nu_h(D_h)) \, \mathrm d \mathcal{H}^2 \\
  &\quad \ge \frac{1-s}{1 + \zeta(4s)} \liminf_{h \to 0} \int_{\partial D_h^{(3 \sigma_h)_h} \cap \Omega'} \psi(\nu_h(D_h^{(3 \sigma_h)_h})) \, \mathrm d \mathcal{H}^2 
  - \frac{\zeta(4s)}{1 + \zeta(3s)} \\ 
  &\quad \ge 2 \int_{\Omega' \cap J_{(r,R)} \cap D^0} \psi_0 (\nu'(r, R)) \, \mathrm d\mathcal{H}^2 
  + \int_{\Omega' \cap \partial^* D} \psi_0 (\nu'(D)) \, \mathrm d\mathcal{H}^2 - C \eps. 
\end{align*} 
Now sending first $\eps$ and then $s$ to $0$ and finally $\omega’ \nearrow \omega$, the monotone convergence theorem gives 
\begin{align*}
  &\liminf_{h \to 0} \int_{\partial^* D_h \cap \Omega} \psi(\nu_h(D_h)) \, \mathrm d \mathcal{H}^2 \\ 
  &\quad \ge 2 \int_{J_{(r,R)} \cap D^0} \psi_0 (\nu'(r, R)) \, \mathrm d\mathcal{H}^1 
  + \int_{\omega \cap \partial^* D} \psi_0 (\nu'(D)) \, \mathrm d\mathcal{H}^1. 
\end{align*} 
\medskip 

\noindent {\bf Lower bound for the bulk part.} 
For the elastic contributions we have to quantify the deviation of $\nabla_h y_h$ from $\SO(3)$. Using the piecewise constant approximation $R_h$, the asymptotic energy can readily be estimated in terms of the limiting strain $G$ along the lines of \cite{FrieseckeJamesMueller:02}. However, due to possible void sets $D_h$, the identification of $G$ is more complicated. For this we use the smooth approximation $\tilde{R}_h$ and an $SBV$ closure argument. 

First we recall $R_h$ from \eqref{eq:Rh-def}, viewed as a function of $x \in \Omega$ independent  of $x_3$, which by \eqref{same convergence of piecewise functions and partition of unity:2} and \eqref{weak convergente of approx. gradients-R} converges to $R$ in $L^2(\omega, \mathbb{R}^{3 \times 3})$ and by \eqref{L2 distance of the rescaled gradients from piecewice rotational function:2} moreover satisfies 
\begin{align}\label{eq:L2-dist-yh-Rh} 
  \| \chi_{\omega_h} (\nabla_h y_h - R_h) \|_{L^2}^2 
   \le C h^2. 
\end{align}

We set 
\begin{align*}
  G_h = \chi_{\omega_h} h^{-1} (R_h^T \nabla_h y_h - \Id). 
\end{align*}
By \eqref{eq:L2-dist-yh-Rh} $G_h$ is $L^2$ bounded and we may pass to a subsequence such that $G_h \rightharpoonup G$ for some $G \in L^2(\Omega, \mathbb{R}^{3 \times 3})$. Since $\Omega \setminus D_h \supset \omega_h \times (0,1)$ we may proceed exactly as in \cite{FrieseckeJamesMueller:02} to see 
\begin{align}\label{eq:lb-W-Q3} 
  \liminf_{h \to 0} h^{-2} \int_{\Omega \setminus D_h} W(\nabla_h y_h) \, \mathrm dx
  \ge \frac{1}{2} \int_{\Omega} Q_3(G) \, \mathrm dx. 
\end{align}
due to the frame invariance of $W$. 

In order to identify $G$ we also define the quantity 
$$ \tilde{G}_h 
   = \chi_{\omega_h} h^{-1} (\tilde{R}_h^T \nabla_h y_h - \Id)  $$
and note that 
$$ G_h - \tilde{G}_{h} 
   = \chi_{\omega_h} h^{-1} (R_h^T \tilde{R}_h - \Id)\tilde{R}_h^T \nabla_h y_h. $$
By \eqref{eq:omegah-Dh} and \eqref{same convergence of piecewise functions and partition of unity} we have 
$$ \chi_{\omega_h} h^{-1} (R_h^T \tilde{R}_h - \Id) 
   \rightharpoonup \chi_{\omega \setminus D} A $$ 
for some $A \in L^2(\omega, \mathbb{R}^{3 \times 3})$ weakly in $L^2(\omega, \mathbb{R}^{3 \times 3})$. Moreover, by \eqref{eq:infty-bd-rR} and \eqref{weak convergente of approx. gradients-R} we have $\tilde{R}_h^T \to R^T$ boundedly in measure and by \eqref{Compactness:4} and \eqref{eq:omegah-Dh} $\chi_{\omega_h} \nabla_h y_h \to \chi_{\omega \setminus D} R$ in $L^2(\Omega, \mathbb{R}^{3 \times 3})$. Thus, 
$$ G_h - \tilde{G}_{h} 
   \rightharpoonup \chi_{\omega \setminus D} A $$ 
weakly in $L^2(\omega, \mathbb{R}^{3 \times 3})$. 

We now determine the upper left $2 \times 2$ submatrix of the weak limit of $\tilde{G}_h$. Let $\Omega' \subset \subset \Omega$. We fix a $z \in \mathbb{R}$ with $|z| < \dist(\Omega', \partial \Omega)$. Denoting by 
$$ \Delta^{(z)} g(x) = \frac{1}{z} \big( g(x', x_3 + z) - g(x) \big) $$ 
the difference quotient $\Delta^{(z)} g : \Omega' \to \mathbb{R}^N$ of a function $g : \Omega \to \mathbb{R}^N$, in particular, we consider the sequence of functions $f_h \in SBV^2(\Omega', \mathbb{R}^3)$, defined by 
$$ f_h(x) 
   = \chi_{\omega_h} (x') h^{-1} \Delta^{(z)} \tilde{R}_h^T(x') y_h(x). $$ 
By \eqref{Compactness:4} and \eqref{eq:omegah-Dh} we have 
\begin{align}\label{eq:Deltaz-yh} 
  \chi_{\omega_h} (x') h^{-1} \Delta^{(z)} y_h(x) 
  &= \chi_{\omega_h} \int_0^1 h^{-1} \partial_3 y_h(x', x_3 + tz) \, \mathrm dt 
\nonumber \\ 
  &\to \chi_{\omega \setminus D} \partial_1 r(x') \wedge \partial_2 r(x') 
\end{align}
in $L^2(\Omega', \mathbb{R}^3)$ and thus, in combination with \eqref{weak convergente of approx. gradients-R},  
\begin{align*} 
  f_h 
  \to \chi_{\omega \setminus D} R^T (\partial_1 r \wedge \partial_2 r) 
  = \chi_{\omega \setminus D} \mathbf{e}_3. 
\end{align*}
in $L^1(\Omega', \mathbb{R}^3)$. 
The absolutely continuous part of the derivative is given by 
\begin{align*} 
  \nabla f_h 
  = \chi_{\omega_h} h^{-1} \big( \partial_1 \tilde{R}_h^T \Delta^{(z)} y_h, \partial_2 \tilde{R}_h^T \Delta^{(z)} y_h, 0 \big) 
  + \chi_{\omega_h} h^{-1} \Delta^{(z)} \tilde{R}_h^T \nabla y_h. 
\end{align*}
By \eqref{weak convergente of approx. gradients-R} and \eqref{eq:Deltaz-yh} the first summand on the right hand side converges to: 
\begin{align}\label{eq:F1h-formula} 
  \chi_{\omega \setminus D} \big( \partial_1 R^T (\partial_1 r \wedge \partial_2 r), \partial_2 R^T (\partial_1 r \wedge \partial_2 r), 0 \big)
\end{align} 
weakly in $L^1(\Omega', \mathbb{R}^{3 \times 3})$. The second summand can be rewritten as 
\begin{align*} 
  \chi_{\omega_h} h^{-1} \Delta^{(z)} \big( \tilde{R}_h^T \nabla y_h - \operatorname{diag}(1,1,h) \big) 
  = \operatorname{diag}(1,1,h) \tilde{G}_h
\end{align*}
and is thus bounded in $L^2(\Omega', \mathbb{R}^{3 \times 3})$, too. As moreover $\mathcal{H}^2 (J_{f_h}) 
\le \mathcal{H}^2 (\partial \omega \cap \partial \omega_h) \le C$ by \eqref{number of bad cubes}, 
the basic closure theorem in $SBV$ (see, e.g., \cite[Theorem~4.7]{AmbrosioFuscoPallara:00}) thus implies 
\begin{align*} 
  \nabla f_h 
  \rightharpoonup \nabla \mathbf{e}_3 
  = 0, 
\end{align*} 
and so the second summand on the right hand side converges to the negative value of \eqref{eq:F1h-formula}, i.e., 
\begin{align*} 
  h^{-1} \chi_{\omega_h} \Delta^{(z)} \tilde{R}_h^T \nabla y_h(x) 
  \rightharpoonup - \chi_{\omega \setminus D} \big( \partial_1 R^T (\partial_1 r \wedge \partial_2 r), \partial_2 R^T (\partial_1 r \wedge \partial_2 r), 0 \big)
\end{align*}
weakly in $L^1(\Omega', \mathbb{R}^{3 \times 3})$. 

This proves that 
\begin{align*} 
\tilde{G}_h'
   &= \chi_{\omega_h} h^{-1} (\tilde{R}_h^T \nabla' y_h - \Id') \\ 
  &\rightharpoonup \chi_{\Omega \setminus D} \big( \tilde{G}(x') - x_3 ( \partial_1 R^T (\partial_1 r \wedge \partial_2 r), R^T (\partial_1 r \wedge \partial_2 r) ) \big) 
\end{align*}
weakly in $L^2(\omega, \mathbb{R}^{3 \times 2})$ for some $\tilde{G} \in L^2(\omega, \mathbb{R}^{3 \times 2})$. So denoting the upper left $2 \times 2$ submatrices of $G$, $A$ and $\tilde{G}$ by $G''$, $A''$ and $\tilde{G}''$, respectively, we arrive at 
\begin{align*}
  G''(x) = \chi_{\omega \setminus D}(x') \big( \tilde{G}''(x') + A''(x') + \tfrac{1}{2} \II(x') + (x_3 - \tfrac{1}{2}) \II(x') \big),  
\end{align*}
where $\II$ is the second fundamental from associated to $r$. 

The remaining part of the proof is analogous to the elastic setting in \cite{FrieseckeJamesMueller:02}: Using $\int_0^1 (x_3 - \tfrac{1}{2}) \mathrm dx_3 = 0$ and $\int_0^1 (x_3 - \tfrac{1}{2})^2 \mathrm dx_3 = \frac{1}{12}$ one computes 
\begin{align*} 
  \int_{\Omega} Q_3(G) \, \mathrm dx 
  &\ge \int_{\omega \setminus D} Q_2(G'') \, \mathrm dx \\ 
  &= \int_{\omega \setminus D} Q_2(\tilde{G}'' + A''  + \tfrac{1}{2}\II) \, \mathrm dx' 
    + \frac{1}{12} \int_{\omega \setminus D} Q_2(\II) \, \mathrm dx' 
\end{align*}
and so \eqref{eq:lb-W-Q3} implies 
\begin{align*}
  \liminf_{h \to 0} h^{-2} \int_{\Omega \setminus D_h} W(\nabla_h y_h) \, \mathrm dx
  \ge \frac{1}{24} \int_{\omega \setminus D} Q_2(\II) \, \mathrm dx'. 
\end{align*}

\section{The upper bound}

We now prove Theorem~\ref{theo:main}(ii). First we construct a recovery sequence for general $r \in SBV_{\rm iso}(\omega)$ with $R = (\partial_1 r, \partial_2 r, \partial_1 r \wedge \partial_2 r)$ and $D \in \mathcal{F}(\omega)$. To this end, we begin with an auxiliary 3d approximation $(w_h, E_h) \in SBV(\Omega;\R^3) \times \mathcal{F}(\Omega)$. 

Fix $\varphi \in C^{\infty}_c(\omega; \R^2)$ and set $f_h : \mathbb{R}^3 \to \R^3$ by
\begin{align}\label{eq: f_h}
f_h(x) 
= (f_h'(x), f_{h,3}(x)) 
= \big( x' - h (x_3 - \tfrac{1}{2}) \varphi(x'), x_3 \big).
\end{align}
For $h$ sufficiently small $f_h(\Omega) \subset \Omega$ and $ f_h| (\mathbb{R}^2 \times (0,1)) $ is a diffeomorphism into  $\mathbb{R}^2 \times (0,1) $. Also fix $d \in W^{1,2}(\omega, \mathbb{R}^3) \cap L^\infty(\omega, \mathbb{R}^3)$. We then define $w_h \in SBV(\Omega; \R^3)$ (see Remark \ref{rmK: coordinate changes SBV}) by setting 
\begin{align}\label{eq: w_h}
  w_h(x) 
  &= r(f_h'(x)) + h (x_3 - \tfrac{1}{2}) \big\{ (\partial_1 r \wedge \partial_2 r)(f_h'(x))  
  + (\nabla' r)(f_h'(x)) \varphi(x') \big\} \\
  &\quad + \tfrac{1}{2} h^2 (x_3 - \tfrac{1}{2})^2 d(x'). \notag
\end{align}
(Viewing $r$ and $R = (\partial_1 r, \partial_2, \partial_1 r \wedge \partial_2)$ as functions on $\Omega$ that do not depend on $x_3$ we can replace the arguments $f_h'(x)$ by $f_h(x)$ here.) 
We first identify the absolutely continuous part of $D w_h$ with the help of \eqref{eq:chain-rule-diffeo}. For $i = 1,2$ we compute 
\begin{align*} 
  \partial_i w_h(x) 
  &= \nabla' r (f_h'(x)) \partial_i f_h'(x) \\
  &\quad+ h (x_3 - \tfrac{1}{2}) \big\{ \nabla' (\partial_1 r \wedge \partial_2 r) (f_h'(x)) \partial_i f_h'(x) \\
  &\qquad\qquad\qquad\quad + \nabla'r (f_h'(x)) \partial_i \varphi 
  + \nabla^2 r(f_h'(x)) [\partial_i f_h'(x), \varphi(x')] \big\} \\
  &\quad+ \tfrac{1}{2} h^2 (x_3 - \tfrac{1}{2})^2 \partial_i d(x'), 
\end{align*} 
where $\partial_i f_h'(x) = \mathbf{e}_i - h (x_3 - \tfrac{1}{2}) \partial_i \varphi(x')$, and so 
\begin{align*} 
  \partial_i w_h(x) 
  &= \partial_i r (f_h'(x)) \\ 
  &\quad + h (x_3 - \tfrac{1}{2}) \big\{ \nabla' (\partial_1 r \wedge \partial_2 r) (f_h'(x)) \mathbf{e}_i 
  + \nabla^2 r(f_h'(x)) [\mathbf{e}_i, \varphi(x')] \big\} \\ 
  &\quad - h^2 (x_3 - \tfrac{1}{2})^2 \big\{ \nabla' (\partial_1 r \wedge \partial_2 r) (f_h'(x)) \partial_i \varphi(x') \\ 
  &\qquad \qquad \qquad \qquad + \nabla^2 r(f_h'(x)) [\partial_i \varphi(x'), \varphi(x')] 
  - \tfrac{1}{2} \partial_i d(x') \big\}. 
\end{align*} 
Since $h^{-1} \partial_3 f_h'(x) = - \varphi(x')$, we obtain  
\begin{align*} 
  h^{-1} \partial_3 w_h(x) 
  &= - \nabla' r (f_h'(x)) \varphi(x') 
  + (\partial_1 r \wedge \partial_2 r)(f_h'(x)) + (\nabla' r)(f_h'(x)) \varphi(x') \\ 
  &\quad - h (x_3 - \tfrac{1}{2}) \big\{ \nabla' (\partial_1 r \wedge \partial_2 r) (f_h'(x)) \varphi(x') \\ 
  &\qquad \qquad \qquad \quad + \nabla^2 r(f_h'(x)) [\varphi(x'), \varphi(x')] - d(x') \big\} \\ 
  &= (\partial_1 r \wedge \partial_2 r)(f_h'(x)) \\ 
  &\quad - h (x_3 - \tfrac{1}{2}) \big\{ \nabla' (\partial_1 r \wedge \partial_2 r) (f_h'(x)) \varphi(x') \\ 
  &\qquad \qquad \qquad \quad + \nabla^2 r(f_h'(x)) [\varphi(x'), \varphi(x')] - d(x') \big\}. 
\end{align*} 

Recall that $\partial_{12}^2 r = \partial_{21}^2 r$ by Remark~\ref{rmk:sym-second-derivatives}. We now make use of the relations 
\begin{align*} 
  \partial_i r \cdot \partial_j (\partial_1 r \wedge \partial_2 r) 
  &= \partial_{ij}^2 r \cdot (\partial_2 r \wedge \partial_1 r) 
   = \II_{ij}, \\ 
  (\partial_1 r \wedge \partial_2 r) \cdot \partial_j (\partial_1 r \wedge \partial_2 r) 
  &= 0 \quad \text{as well as} \\ 
  \partial_i r \cdot \partial^2_{ij} r 
  &= 0, 
\end{align*} 
for $i, j \in \{1,2\}$. (These follow from $\partial_j (R^T R) = 0$ and the product rule for bounded $SBV$ functions.) They show that 
\begin{align*} 
  R^T \nabla' (\partial_1 r \wedge \partial_2 r)
  = \begin{pmatrix}
      \II_{11} & \II_{12} \\ 
      \II_{21} & \II_{22} \\ 
      0 & 0 
     \end{pmatrix} 
   \quad\text{and}\quad 
 R^T \nabla^2 r [a, b] 
  = \begin{pmatrix}
      0 \\ 0 \\ - a^T \II b 
     \end{pmatrix} 
\end{align*} 
for $a, b \in \R^2$. So our calculations imply 
\begin{align}\label{eq:Gh-Ah-Bh-def} 
\begin{split}
  R^T(f_h'(x)) \nabla_h w_h(x) 
  &= \Id + h (x_3 - \tfrac{1}{2}) G_h(x) \\ 
  &:= \Id + h (x_3 - \tfrac{1}{2}) A_h(x) + h^2 (x_3 - \tfrac{1}{2})^2 B_h(x) 
\end{split}
\end{align} 
with 
\begin{align*} 
  A_h(x) 
  &=\begin{pmatrix}[1.3]
      \II  (f_h'(x)) & \vline & - \II(f_h'(x)) \varphi(x') + ( R^T(f_h'(x)) d(x') )'  \\ 
     \hline 
      - ( \varphi(x') )^T \II(f_h'(x))  & \vline & (\varphi(x'))^T \II(f_h'(x)) \varphi(x') + ( R^T(f_h'(x)) d(x') )_3    
     \end{pmatrix}, \\ 
  B_h(x) 
  &= \begin{pmatrix}[1.3]
      - \II  (f_h'(x)) \nabla' \varphi(x') + \tfrac{1}{2} ( R^T(f_h'(x)) \nabla' d(x') )' & \vline & 0 \\ 
     \hline 
      (\varphi(x'))^T \II(f_h'(x)) \nabla' \varphi(x')  + \tfrac{1}{2} ( R^T(f_h'(x)) \nabla' d(x') )_3 & \vline & 0    
     \end{pmatrix}. 
\end{align*} 

We note that since $f_h$ is a diffeormorphism on $\Omega$ with $f_h \to \id$ and $\nabla f_h \to \Id$ uniformly as $h \to \infty$ we have that $G_h \to G$ in $L^2(\Omega; \R^{3 \times 3})$, where 
\begin{align*} 
  G(x) 
  = \begin{pmatrix}[1.3]
      \II & \vline & - \II \varphi + ( R^T d )'  \\ 
     \hline 
      - \varphi^T \II & \vline & \varphi^T \II \varphi + ( R^T d )_3    
     \end{pmatrix} (x'). 
\end{align*} 
But then we also have 
\begin{equation}\label{eq:Gh-L1-conv} 
  h^{-2} W \big( \Id + h (x_3 - \tfrac{1}{2}) G_h \big) 
  \to \tfrac{1}{2} Q_3 \big( (x_3 - \tfrac{1}{2}) G \big) 
  = \tfrac{1}{2} (x_3 - \tfrac{1}{2})^2 Q_3(G) 
\end{equation}
in $L^1(\Omega; \R^{3 \times 3})$. To see this, we note first that $W(\Id) = 0$ and $\nabla W(\Id) = 0$. On the one hand we deduce form this that $h^{-2} W (\Id + h (x_3 - \tfrac{1}{2}) G_h) \to Q_3 ((x_3 - \tfrac{1}{2})G)$ in measure by Taylor expanding $W$ around $\Id$. On the other hand, it implies that there exist $c, C > 0$ such that $W(\Id + X) \le C|X|^2$ whenever $|X| \le c$. By the growth condition $W(X) \le C(1 + |X|^2)$ we then also have 
\[ W(\Id + X) 
   \le C|X|^2 \] 
for all $X \in \R^{3 \times 3}$. So 
\[ h^{-2} W \big( \Id  + h (x_3 - \tfrac{1}{2}) G_h(x) \big) 
   \le | G_h(x) |^2 \]
and since $(G_h)$ is convergent in $L^2$ this proves that $h^{-2} W(\Id  + h G_h)$ is uniformly integrable. So \eqref{eq:Gh-L1-conv} follows form the Vitali convergence theorem.

Now we consider the jump part $J_{w_h}$ of $w_h$. We view $r$ and $R = (\nabla'r, \partial_1 \wedge \partial_2 r)$ as $SBV$ functions in $\Omega$ that are independent of $x_3$ with (cylindrical) jump sets $J_r, J_R \subset \Omega$, respectively, and a normal $\nu = (\nu',\nu_3) = (\nu',0)$ given $\mathcal{H}^2$-a.e.\ on $J_r \cup J_R$. Then (cp.\ Section~\ref{sec:NotationPreliminaries}) the jump set of $w_h$ satisfies 
\[ J_{w_h} \subset f_h^{-1}(J_r \cup J_R) \]
and a normal field on $J_{w_h}$ is given by 
\[ \nu(w_h) (x) 
   = \frac{(\nabla f_h(x))^T \nu(f_h(x))}{|(\nabla f_h(x))^T \nu(f_h(x))|} 
   \qquad \text{for } \mathcal{H}^2\text{-a.e. } x \in J_{w_h},  
\]
see \eqref{eq:normal-diffeo}, where $\nabla f_h(x) = \Id - 
   h \begin{pmatrix}
      (x_3 - \tfrac{1}{2}) \nabla' \varphi(x') & \vline & \varphi(x') \\ 
     \hline 
     0 & \vline & 0   
     \end{pmatrix}$.  
It follows that 
\begin{align*}
  &(\nabla f_h(x))^T \nu(f_h(x)) \\ 
  &~~= \big( \nu'(f_h(x))- h (x_3 - \tfrac{1}{2}) (\nabla' \varphi(x'))^T \nu'(f_h(x)), - h \varphi(x') \cdot \nu'(f_h(x)) \big) 
\end{align*}
and so, since $|\nu'| = |\nu| = 1$, 
\begin{align*}
  \nu(w_h) (x) 
  &= \big( \nu'(f_h(x)) + O(h), - h \varphi(x') \cdot \nu'(f_h(x)) + O(h^2) \big). 
\end{align*}
Rescaling and making use of $\varphi(x') = \varphi(f_h(x)) + O(h)$, when extended to $\Omega$ as a function independent of $x_3$, we arrive at
\begin{align}\label{eq:normal-wh}
  \nu_h(w_h) (x) 
  &= \big( \nu'(f_h(x)), - \varphi(f_h(x)) \cdot \nu'(f_h(x)) \big) 
  + O(h) 
\end{align}
for $\mathcal{H}^{2}$-a.e.\ $x \in J_{w_h}$. 

We finally also view $D$ as a (cylindrical) set of finite perimeter in $\Omega$ and define $E_h = f_h^{-1} (D)$. Again we denote by $\nu = (\nu',0)$ the unit outer normal to $D$. As $\chi_{E_h}(x) = \chi_D(f_h(x))$, the same reasoning that led to \eqref{eq:normal-wh} now shows that $\partial^* E_h \cap \Omega = f_h^{-1}(\partial^* D) \cap \Omega$ and that the rescaled unit outer normal to $E_h$ is given by 
\begin{align*}
  \nu_h(E_h) (x) 
  &= \big( \nu'(f_h(x)), - \varphi(f_h(x)) \cdot \nu'(f_h(x)) \big) 
  + O(h) 
\end{align*}
for $\mathcal{H}^{2}$-a.e.\ $x \in \partial^*E_h \cap \Omega$. 

Now we notice that both 
\begin{align}\label{measure-fh-est}
\begin{split}
  f_h \# \mathcal{H}^2 \restrict J_{w_h} 
  &\le (1 + O(h)) \mathcal{H}^2 \restrict (J_r \cup J_R) \quad \text{and} \\ 
  f_h \# \mathcal{H}^2 \restrict \partial^* E_h 
  &\le (1 + O(h)) \mathcal{H}^2 \restrict \partial^* D. 
\end{split}
\end{align}
This follows since for any Borel set $A \subset \Omega$ and $S_h = J_{w_h}$ or $S_h = \partial^* E_h$ 
\begin{align*}
  f_h \# \mathcal{H}^2 \restrict S_h (A) 
  &= \mathcal{H}^2 \big(S_h \cap f_h^{-1}(A) \big) \\ 
  &= \mathcal{H}^2 \big( f_h^{-1} \big( f_h(S_h) \cap A \big) \big) \\ 
  &\le (1 + O(h)) \mathcal{H}^2 \big( f_h(S_h) \cap A \big)
\end{align*}
where $f_h(S_h) \subset J_r \cup J_R$, respectively, $f_h(S_h) = \partial^* D$ and we have used the fact that $f_h^{-1}$ is a Lipschitz mapping with Lipschitz constant bounded by $1 + O(h)$. 

Combining \eqref{eq:normal-wh}, \eqref{eq:normal-wh} and \eqref{measure-fh-est} we find 
\begingroup
\allowdisplaybreaks
\begin{align*}
  &2 \int_{J_{w_h} \cap E_h^0} \psi \big( \nu_h(w_h)(x) \big) \, \mathrm d \mathcal{H}^2(x) 
  + \int_{\Omega \cap \partial^* E_h} \psi \big( \nu_h(E_h)(x) \big) \, \mathrm d \mathcal{H}^2(x) \\ 
  &~~= 2 \int_{\Omega \cap E_h^0} \psi \big( \nu'(f_h(x)), - \varphi(f_h(x)) \cdot \nu'(f_h(x)) \big) \, \mathrm d \mathcal{H}^2 \restrict J_{w_h }(x) \\ 
  &\qquad + \int_{\Omega} \psi \big( \nu'(f_h(x)), - \varphi(f_h(x)) \cdot \nu'(f_h(x)) \big) \, \mathrm d \mathcal{H}^2 \restrict \partial^* E_h(x)
  + O(h) \\ 
  &~~= 2 \int_{\Omega \cap D^0} \psi \big( \nu'(y), - \varphi(y) \cdot \nu'(y) \big) \, \mathrm d f_h \# \mathcal{H}^2 \restrict J_{w_h} (y) \\ 
  &\qquad +  \int_{\Omega} \psi \big( \nu'(y), - \varphi(y) \cdot \nu'(y) \big) \, \mathrm d f_h \# \mathcal{H}^2 \restrict \partial^* E_h (y) 
  + O(h) \\
  &~~\le 2 \int_{\Omega \cap D^0} \psi \big( \nu'(y), - \varphi(y) \cdot \nu'(y) \big) \, \mathrm d \mathcal{H}^2 \restrict (J_r \cup J_R) \\ 
  &\qquad +  \int_{\Omega} \psi \big( \nu'(y), - \varphi(y) \cdot \nu'(y) \big) \, \mathrm d \mathcal{H}^2 \restrict \partial^* D
  + O(h),  
\end{align*}
\endgroup
where we have used $E_h^0 = f_h^{-1} (D^0)$ in the second step. 

Together with \eqref{eq:Gh-Ah-Bh-def}, \eqref{eq:Gh-L1-conv}, the fact that $\chi_{E_h} \to \chi_D$ boundedly in measure and the frame invariance of $W$ we finally find that 
\begin{align*}
  &\limsup_{h \to 0} \bigg[ h^{-2}\int_{\Omega \setminus E_h} W(\nabla_h w_h(x)) \, \mathrm{d}x \\
  &\qquad\qquad
  + 2 \int_{J_{w_h} \cap E_h^0} \psi \big( \nu_h(w_h)(x) \big) \, \mathrm d \mathcal{H}^2(x) + \int_{\Omega \cap \partial^* E_h} \psi \big( \nu_h(E_h)(x) \big) \, \mathrm d \mathcal{H}^2(x) \bigg] \\  
  &~~\le \int_{\Omega \setminus D} \tfrac{1}{2} (x_3 - \tfrac{1}{2})^2 Q_3(G(x)) \, \mathrm{d}x 
  + 2 \int_{(J_r \cup J_R) \cap D^0} \psi \big( \nu'(x), - \varphi(x) \cdot \nu'(x) \big) \, \mathrm d \mathcal{H}^2 \\ 
  &\qquad \qquad +  \int_{\partial^* D \cap \Omega} \psi \big( \nu'(x), - \varphi(x) \cdot \nu'(x) \big) \, \mathrm d \mathcal{H}^2 \\ 
  &~~= \frac{1}{24} \int_{\omega \setminus D} Q_3(G(x')) \, \mathrm{d}x 
  + 2 \int_{(J_r \cup J_R) \cap D^0} \psi \big( \nu'(x'), - \varphi(x' \cdot \nu'(x') \big) \, \mathrm d \mathcal{H}^1 (x) \\ 
  &\qquad \qquad +  \int_{\partial^* D \cap \omega} \psi \big( \nu'(x'), - \varphi(x') \cdot \nu'(x') \big) \, \mathrm d \mathcal{H}^1 (x),  
\end{align*}
where we passed to the planar setting in the last step by integrating with respect to $x_3$. 

We abbreviate $S = J_r \cup J_R \cup (\partial^* D \cap \omega)$ and select $\bar{\nu}_3 \in L^{\infty}_{\mathcal{H}^2 \restrict S}(\omega)$ such that 
\[ \psi(\nu'(y), \bar{\nu}_3(y)) 
   = \psi_0(\nu'(y)) 
   \quad\text{for $\mathcal{H}^1$-a.e.\ } y \in S. \] 
Given $\eps > 0$ we then choose $\varphi$ such that 
\[ \int_S | \bar{\nu}_3(y') + \varphi(y') \cdot \nu'(y')| \, \mathrm d\mathcal{H}^1
   \le \tfrac{1}{3} \eps. \] 
(Let $\mu \in C^{\infty}_c(\omega; \R^2)$ with $\|\mu\|_\infty \le 2$ such that $\int_{\omega} |\nu' - \mu| \, \mathrm d\mathcal{H}^1 \restrict S < \eps / (6\| \bar{\nu}_3 \|_{L^\infty})$ and $\theta \in C^{\infty}_c(\omega)$ such that $\int_{\omega} |\bar{\nu}_3 - \theta| \, \mathrm d \mathcal{H}^1 \restrict S < \eps / 12$, set $\varphi = - \theta \mu$ and notice $| \bar{\nu}_3 - \theta \mu \cdot \nu'| \le |\bar{\nu}_3 (\nu' - \mu) \cdot \nu'| + |(\bar{\nu}_3 - \theta) \mu \cdot \nu'|$.) Then we choose $d$ such that 
\[ \int_\omega |Q_3(G(x')) - Q_2(\II(x'))| \, \mathrm d x 
\le \tfrac{1}{3} \eps. \]
With these choices we have
\begin{align}\label{eq:whEh-est}
    &\limsup_{h \to 0} \bigg[ h^{-2} \int_{\Omega \setminus E_h} W(\nabla_h w_h(x)) \, \mathrm{d}x \nonumber \\
  &\qquad\qquad
  + 2 \int_{J_{w_h} \cap E_h^0} \psi \big( \nu_h(w_h)(x) \big) \, \mathrm d \mathcal{H}^2(x) + \int_{\Omega \cap \partial^* E_h} \psi \big( \nu_h(E_h)(x) \big) \, \mathrm d \mathcal{H}^2(x) \bigg] \nonumber \\  
  &~~\le \frac{1}{24} \int_{\omega \setminus D} Q_3(\II(x')) \, \mathrm{d}x \\ 
  &\qquad \qquad 
  + 2 \int_{(J_r \cup J_R) \cap D^0} \psi_0 ( \nu'(x') ) \, \mathrm d \mathcal{H}^1 
  +  \int_{\partial^* D \cap \omega} \psi_0 ( \nu'(x') ) \, \mathrm d \mathcal{H}^1 + \eps. \nonumber 
\end{align}

In a second step we can now construct recovery sequences for $(y, D)$. To this end we apply the Relaxation Theorem~\ref{theo:bulk-relax} to the functional $\mathcal{E}_{h} : W^{1,2}(\Omega; \mathbb{R}^3) \times \mathcal{F}(\Omega) \rightarrow  \mathbb{R}$ for fixed $h$, which can be written as 
\begin{align*}
  \mathcal{E}_{h}(y,D) 
  = \int_{\Omega \setminus D} W_h(\nabla y) \, \mathrm{d}x  
   + \int_{\Omega \cap \partial^* D} \psi_h (\nu(D)) \, \mathrm{d}\mathcal{H}^2,  
\end{align*}
where $W_h(X) = h^{-2} W(X', h^{-1} X_{\cdot 3})$ and $\psi_h(x) = \psi(x', h^{-1} x_3)$. Theorem~\ref{theo:bulk-relax}(ii) provides a sequence $(w_{h,k}, E_{h,k}) \in W^{1,2}(\Omega; \R^{3}) \times \mathcal{F}(\Omega)$ such that $w_{h,k} \to w_h$ in $L^2(\Omega;\R^3)$, $\| w_{h,k} \|_{L^\infty} \to \| w_h \|_{L^\infty}$, $\chi_{E_{h,k}} \to \chi_{E_h}$ in $L^1(\Omega)$ and 
\begin{align*}
  \limsup_{k \to \infty} \mathcal{E}_{h}(w_{h,k}, E_{h,k}) 
  &\le \int_{\Omega \setminus E_h} W(\nabla_h w_h) \, \mathrm{d}x 
  + 2 \int_{J_{w_h} \cap E_h^0} \psi \big( \nu_h(w_h) \big) \, \mathrm d \mathcal{H}^2 \\ 
  &\qquad + \int_{\Omega \cap \partial^* E_h} \psi \big( \nu_h(E_h) \big) \, \mathrm d \mathcal{H}^2.  
\end{align*}
Now we can choose a diagonal sequence $y_h = w_{h, k_h}$, $D_h = E_{h, k_h}$ such that 
\begin{align*}
  \limsup_{h \to 0} \mathcal{E}_{h}(y_h, D_h) 
  &\le\frac{1}{24} \int_{\omega \setminus D} Q_3(\II(x')) \, \mathrm{d}x 
  + 2 \int_{(J_r \cup J_R) \cap D^0} \psi_0 ( \nu'(x') ) \, \mathrm d \mathcal{H}^1 \\ 
  &\qquad +  \int_{\partial^* D \cap \omega} \psi_0 ( \nu'(x') ) \, \mathrm d \mathcal{H}^1 + \eps. 
\end{align*}
As $\eps > 0$ was arbitrary, the construction is complete. 

\medskip 

We finally show that, in case $(y, D) \in  SBV^{2,2}_{\rm iso}(\omega) \times \mathcal{F}(\omega)$ satisfies \eqref{eq: Minkowski content} (with $ r $ replaced by $ y $), the recovery sequence $(y_h, D_h)$ can be chosen such that $(D_h)$  satisfies the rescaled $ \psi $-minimal droplet assumption \eqref{eq:min-drop-D} with a universal function $\zeta_0$. We abbreviate $ J = J_{(r , \nabla r)} $ and define 
$$F_t = \{ x \in \mathbb{R}^2 : \dist_{\psi_0^\circ}(J \cup D, x) \leq s \} \quad \textrm{for $ s > 0 $}.$$ 
Let $ \epsilon > 0 $ and we apply \eqref{eq: Minkowski content} to choose $ 0 < \tau < \epsilon $ so that
\begin{align*} 
  \frac{\mathcal{L}^2 (\omega \cap F_{\tau} \setminus D)}{\tau} 
  \le 2 \int_{J \cap D^0} \psi_0 (\nu(J)) \, \mathrm{d}\mathcal{H}^1
  + \int_{\omega \cap \partial^* D} \psi_0 (\nu(D)) \, \mathrm{d}\mathcal{H}^1 + \eps. 
\end{align*}
Applying the anisotropic coarea formula in \eqref{eq: Coarea} with $ u = -\dist_{\psi_0^\circ}(J \cup D, \cdot) $ and noting that $ \psi_0(\nabla \dist_{\psi_0^\circ}(J \cup D,x)) = 1 $ for $ \mathcal{L}^2 $ a.e.\ $ x \in \mathbb{R}^2 \setminus \overline{J \cup D} $, we infer that 
\begin{align*}
  \mathcal{L}^2 (\omega \cap F_\tau \setminus D) 
  \ge \int_{\omega \cap F_\tau \setminus D}\psi_0(\nabla \dist_{\psi_0^\circ}(J \cup D,\cdot))\, \mathrm{d}\mathcal{L}^2
  = \int_{0}^{\tau} \psi_0(D \rchi_{F_t})(\omega)\,\mathrm{d}t 
\end{align*}
and we conclude that $\psi_0(D \rchi_{F_t})(\omega) \le \tau^{-1} \mathcal{L}^2 (\omega \cap F_\tau \setminus D)$ for uncountably many $t \in (0, \tau)$.  We fix one of those $ t $, which in addition satisfies $\mathcal{H}^1(\partial^* F_t \cap \partial \omega) = 0$. (This expression can be positive for at most countably many values of $t$.)  Approximating the set $ F_t $ by a sequence of sets with smooth boundaries (see \cite[Theorem~2.5]{SantilliSchmidt:21}), we may choose an open set $F'_t \subset \mathbb{R}^2$ with smooth boundary such that $F_t \subset F'_t \subset F_\tau$, $ \mathcal{H}^1(\partial F'_t \cap \partial \omega) =0 $ and
\begin{align}\label{eq:Dt-D-est}
\begin{split}
  \int_{\omega \cap \partial F'_t} \psi_0(\nu(F'_t)) \, \mathrm{d}\mathcal{H}^1 
  &\le \int_{\omega \cap \partial^* F_t} \psi_0(\nu(F_t)) \, \mathrm{d}\mathcal{H}^1 + \eps \\ 
  &\le 2 \int_{J \cap D^0} \psi_0 (\nu(J)) \, \mathrm{d}\mathcal{H}^1
  + \int_{\omega \cap \partial^* D} \psi_0 (\nu(D)) \, \mathrm{d}\mathcal{H}^1 + 2 \eps. 
\end{split}
\end{align}
We recall the maps $ f_h $ and $w_h $ from \eqref{eq: f_h} and \eqref{eq: w_h}, we define the sets $$E_{t,h} = f_h^{-1}( F'_t \times (0,1))$$ and we notice that $\mathcal{H}^2(\partial E_{t,h} \cap \partial \Omega) = 0$ and $J_{w_h} \subset E_{t,h} \cap \Omega$ by Remark \ref{rmK: coordinate changes SBV}. By possibly multiplying with smooth cut-off functions that vanish on $J_{w_h}$ but are equal to $1$ outside $E_{t,h}$ and such that still $w_h \to r$ in $L^1$, we may assume without loss of generality that $w_h \in W^{1,2}(\Omega;\R^3)$. The arguments leading to \eqref{eq:whEh-est} now imply that 
\begin{align}
  \limsup_{h \to 0} h^{-2} \int_{\Omega \setminus E_{t,h}} W(\nabla_h w_h(x)) \, \mathrm{d}x 
  &\le \frac{1}{24} \int_{\omega \setminus D} Q_3(\II(x')) \, \mathrm{d}x + \eps \label{eq:ela-est}
\shortintertext{and} 
  \limsup_{h \to 0} \int_{\partial^* E_{t,h} \cap \Omega} \psi \big( \nu_h(E_{t,h})(x) \big) \, \mathrm d \mathcal{H}^2(x) 
  &\le 
  \int_{\partial^* F'_t \cap \omega} \psi_0 ( \nu'(x') ) \, \mathrm d \mathcal{H}^1 + \eps. \label{eq:Eth-Dt-est}
\end{align}

In order to check for the minimal droplet assumption we rescale the sets $E_{t,h}$ and also shift the plate from $\Omega_h$ to the $x_3$-symmetric domain $\Omega^\ast_h = \omega \times (-\frac{h}{2}, \frac{h}{2})$. We consider the ($h$ independent) function $g : \R^3 \to \R^3$, 
\[ g(x) 
    = \big( x' - x_3 \varphi(x'), x_3 \big)
\]
and observe that $ g(x', h(x_3 - \frac{1}{2})) = f_h(x) $ for $ x \in \mathbb{R}^3 $, $ g(\Omega^{\ast}_h) \subset \Omega^{\ast}_h $ and there is an $h_0 > 0$ such that $g : \mathbb{R}^2 \times (-\frac{h}{2}, \frac{h}{2}) \rightarrow \mathbb{R}^2 \times (-\frac{h}{2}, \frac{h}{2}) $ is a diffeomorphism for each $0 < h \le h_0$ (as $ f_h $ is a diffeomorphism for each $ 0 < h \leq h_0 $). We define 
\[ G_{t,h} 
    = g^{-1} \big( F'_t \times (-\tfrac{h}{2}, \tfrac{h}{2}) \big) = \bigg\{\Big(x', h\Big(x_3- \frac{1}{2}\Big)\Big) : x \in E_{t,h}\bigg\}. 
\]
For every $\frac{h_0}{2} \le a < b \le \frac{h_0}{2}$ we notice that
$$ g^{-1}((\partial F'_t \cap \omega) \times (a,b)) = \partial G_{t, h_0} \cap (\omega \times (a,b)) $$ 
and we infer that there exists a constant $c_0 > 0$ (depending only on $ g $ and $ \psi $) such that 
\[ c_0^{-1} (b - a) \mathcal{H}^1 (\omega \cap \partial F'_t) 
   \le \psi(D\rchi_{G_{t,h_0}}) (U_{a,b}) 
   \le c_0 (b - a) \mathcal{H}^1 (\omega \cap \partial F'_t). \]
It follows that
\begin{align}\label{eq:x3-dep-meas} 
  \int_{\Omega_{h+chs}^\ast \cap \partial G_{t,h_0}} \psi(\nu(G_{t,h_0})) \, \mathrm{d}\mathcal{H}^2 
  \le (1 + c_0^2 c s) \int_{\Omega_{h}^\ast \cap \partial G_{t,h_0}} \psi(\nu(G_{t,h_0})) \, \mathrm{d}\mathcal{H}^2
\end{align}
for every $ c > 0 $ and for every $ 0 < s < 1 $.

For each $0 < \eta < 1$ let $\psi_\eta$ be a uniformly convex and smooth norm with 
\[ (1 + \eta)^{-1} \psi_\eta 
    \le \psi 
    \le (1 + \eta) \psi_\eta.
\]
We denote the tubular neighborhood of a set $X$ of radius $ r > 0 $ with respect to $ \psi_\eta^\circ $ by $B^\eta_r(X)$.  
Since $G_{t,h_0}^{(sh)}  \subset B^\eta_{(1+\eta)sh}(G_{t,h_0})$ and $G_{t,h} \cap (\Omega^\ast_h)^- = G_{t, h_0} \cap (\Omega^\ast_h)^-$, we can apply the inequality in \eqref{eq: Steiner inequality} in the Appendix \ref{appendix:Steiner} with $ G = G_{t, h_0} $ and $ \phi = \psi_\eta $ to infer for $ 0 < h < h_0 $ that
\begin{align*} 
 & \mathcal{L}^3 \big( (G_{t,h}^{(sh)} \setminus G_{t,h}) \cap (\Omega_h^\ast)^- \big) \\
  & \qquad \leq \mathcal{L}^3 \big( (G_{t,h_0}^{(sh)} \setminus G_{t,h_0}) \cap (\Omega_h^\ast)^- \big)  \\
  & \qquad  \leq \mathcal{L}^3 \big( (B^\eta_{(1+\eta)sh}(G_{t,h_0}) \setminus G_{t,h_0}) \cap (\Omega_h^\ast)^- \big) \\
  &\qquad \le \bigg( 1 + \frac{(1+\eta)sh}{\gamma_1(\psi_\eta) \gamma_2(G_{t,h_0})} \bigg)^{2} (1+\eta)sh \int_{B^\eta_{(1+\eta)sh}((\Omega_h^\ast)^-)
  \cap \partial G_{t,h_0}} \psi_\eta(\nu(G_{t,h_0})) \, \mathrm{d}\mathcal{H}^2 \\ 
  &\qquad \le \big( 1 +  \bar{C}(\eta)C(t)  sh + 3 \eta \big) sh \int_{\Omega_{h + c_1hs}^\ast \cap \partial G_{t,h_0}} \psi(\nu(G_{t,h_0})) \, \mathrm{d}\mathcal{H}^2
\end{align*}
for every $ 0 < s < s_0 $, where $ c_1 > 0 $ and $ 0 < s_0 < 1 $ are constant depending on $ \psi $, and for a $t$-dependent constant $C(t)$ and an $\eta$-dependent constant $\bar{C}(\eta)$.  Combining with \eqref{eq:x3-dep-meas} we get 
\begin{align*} 
  \mathcal{L}^3 \big( (G_{t,h}^{(sh)} \setminus G_{t,h}) \cap (\Omega_h^{\rm s})^- \big) 
  &\le \big( 1 + c_2 s +  c_2 C(t)  \bar{C}(\eta) sh + 3 \eta\big) sh \int_{\Omega_{h}^{\ast}\cap\partial^* G_{t,h}} \psi(\nu(G_{t,h})) \, \mathrm{d}\mathcal{H}^2
\end{align*}
for every $ 0 < s < s_0 $ and for a constant $ c_2 > 0 $ depending on $ \psi $ and $ g $.
Choosing $\eta = \eta(h) \searrow 0$ so slowly that $\bar{C}(\eta(h)) h \to 0$ we see that these sets satisfy the minimal droplet assumption for a suitable $\zeta_0$.

It follows from \eqref{eq:Eth-Dt-est}, \eqref{eq:Dt-D-est} and \eqref{eq:ela-est} that 
\begin{align*}
  \limsup_{\tau \to 0} \limsup_{h \to 0} \mathcal{E}_h(w_h, E_{t,h}) 
  \le \frac{1}{24} \int_{\omega \setminus D} Q_3(\II(x')) \, \mathrm{d}x 
  + \int_{\partial^* D_t \cap \omega} \psi_0 ( \nu'(x') ) \, \mathrm d \mathcal{H}^1 
  + 4 \eps. 
\end{align*} 
Since $\eps > 0$ was arbitrary we may pass to a diagonal sequence $(y_h, D_h) \in W^{1,2}(\Omega; \mathbb{R}^3) \times \mathcal{F}(\Omega)$ by choosing $\tau = \tau_h \to 0$ sufficiently slowly such that $y_h \to r$ in $L^1(\Omega;\R^3)$, $\chi_{D_h} \to \chi_D$ in $L^1(\Omega)$ and $\limsup_{h \to 0} \mathcal{E}_h(y_h, D_h) \le \mathcal{E}(r,D)$. 
\hfill $\Box$

\begin{appendix}

\section{Auxiliary results}

\subsection*{Minkowski content}

Let $ \phi $ be an arbitrary norm on $ \mathbb{R}^n $ and $\phi^\circ$ its dual norm. The set $\mathcal{W}^\phi = \{ x \in \mathbb{R}^n : \psi^\circ(x) \leq 1 \}$ is a centrally symmetric convex body called the Wulff shape of $ \phi $. If $ A \subset \mathbb{R}^n $ we set $ B^\phi_r(A) = \{x : \mathbb{R}^n : \dist_\phi(x, A) \leq r \} $ for every $ r > 0 $.

We state a formula for the outer Minkowski content of sufficiently regular sets. In the global anisotropic setting this has been obtained in \cite[Theorem~4.4]{LussardiVilla:16}.  

\begin{theorem}\label{theo:Minkowski-density-lb} 
	Suppose $E \subset \R^n$ is a closed set such that $\partial E$ is $\mathcal{H}^{n-1}$ rectifiable and satisfies the density condition 
	\begin{align*}
		\mu(B_r(x)) \ge \gamma r^{n-1} 
		\quad \forall\, x \in \partial E ~ \forall\, r \in (0,1) 
	\end{align*} 
	for a constant $\gamma > 0$ and a finite measure $\mu$ on $\R^n$ with $\mu \ll \mathcal{H}^{n-1}$. If $A \subset \R^n$  is a Borel set with $\mathcal{H}^{n-1}(\partial A \cap \partial E) = 0$, then 
	\begin{align*}
		&\lim_{r \to 0} \frac{\mathcal{L}^n \big( ( B^\phi_r(E)  \setminus E) \cap A \big)}{r} \\ 
		&~~=2 \int_{\partial E \cap E^0 \cap A} \phi (\nu(\partial E)) \, \mathrm{d}\mathcal{H}^{n-1}
		+ \int_{\partial^* E \cap A} \phi (\nu(E)) \, \mathrm{d}\mathcal{H}^{n-1}. 
	\end{align*}
\end{theorem}

The proof follows from the global result \cite[Theorem~4.4]{LussardiVilla:16} (for $A = \R^n$) and a local version of the  lower estimate \cite[Eq.~(4.8)]{LussardiVilla:16}. Details are contained in the proof of \cite[Lemma~3.7]{SantilliSchmidt:21}.

\subsection*{Anisotropic Steiner formula}\label{appendix:Steiner}

Suppose now $ \phi \in C^2(\R^n \setminus \{0\}) $ is a uniformly convex norm. Then the dual norm $ \phi^\circ $ is uniformly convex as well and the Wulff shape $\mathcal{W}^\phi$ is a centrally symmetric uniformly convex body. 

Suppose $ G \subset \mathbb{R}^n $ is an open set with $ C^2 $-boundary. The exterior $ \phi $-anisotropic normal of $ G $ is the map $ \nu^\phi(G): \partial G \rightarrow \partial \mathcal{W}^\phi $ defined as
\begin{equation*}
\nu^\phi(G)(x) = \nabla \phi(\nu(G)(x)) = \qquad \textrm{for $ x \in \partial G $}.
\end{equation*}
One observes that the tangential derivative $ D \nu^\phi(G)(x) $ of $ \nu^\phi_G $ at $ x $ is an endomorphism of $ \Tan(\partial G, x) $ (see \cite[Remark 2.25]{DeRosaKolasinskiSantilli}) that has $ n $ (counted with multiplicity) eigenvalues $ \kappa^\phi_{G,1}(x) \leq \ldots \leq \kappa^\phi_{G,n-1}(x) $. These numbers are the $ \phi $-principal curvatures of $ \partial G $ at $ x $ (with respect to the anisotropic exterior normal $ \nu^\phi(G)(x) $); see \cite[Definition 2.26]{DeRosaKolasinskiSantilli}. We define the positive continuous functions
\begin{alignat*}{2}
    \rho^\phi_G(x) 
    &= \sup\big\{  s > 0 : \dist_\phi\big(x - s \nu^\phi(G)(x), \mathbb{R}^n \setminus G\big) = s \big\} \quad  
     & & \textrm{for $ x \in \partial G $,}\\ 
	\rho_G(x) 
	&= \sup\big\{  s > 0 : \dist\big(x - s \nu(G)(x), \mathbb{R}^n \setminus G\big) = s \big\} \quad 
	 & & \textrm{for $ x \in \partial G $,}\\
    \rho^\phi(\eta) 
    &= \sup\big\{s > 0:  \dist_\phi\big(\eta - s \nu^\phi(B_1)(\eta), \mathbb{R}^n \setminus B_1\big) =s \big\} \quad 
     & & \textrm{for $ \eta \in \partial B_1 $}
\end{alignat*}
and we observe by a scaling argument that 
\begin{equation*}
		\inf_{x \in \partial G} \rho^\phi_G(x) \geq \Big( \inf_{x \in \partial G} \rho_G(x) \Big) \cdot \Big( \inf_{\eta \in \partial B_1} \rho^\phi(\eta)\Big).
\end{equation*}
Setting $ \gamma(\phi, G) = \inf_{x \in \partial G} \rho^\phi_G(x) $, $ \gamma_1(\phi) = \inf_{\eta \in \partial B_1} \rho^\phi(\eta) $ and $\gamma_2(G) = \inf_{x \in \partial G} \rho_G(x) $, we use \cite[Lemma 2.34]{DeRosaKolasinskiSantilli} to conclude that 
\begin{equation*}
	\kappa^\phi_{G,i}(x) \leq \frac{1}{\gamma(\phi, G)} \leq \frac{1}{\gamma_1(\phi) \gamma_2(G)} \qquad \textrm{for every $ x \in \partial G $ and for every $ i = 1, \ldots , n $.}
\end{equation*}
Now for every Borel set $ E \subset \mathbb{R}^n $ we can use the disintegration formula \cite[Theorem 3.18]{HugSantilli}  to obtain
\begin{flalign}\label{eq: Steiner inequality}
	\mathcal{L}^n\big((B^\phi_r(G) \setminus G) \cap E\big) & \leq \int_{\partial G} \phi(\nu(G)(x))\int_{0}^r \rchi_E(x + t \nu^\phi(G)(x))\prod_{i=1}^{n-1}(1 + t\kappa^\phi_{G,i}(x))\, \mathrm{d}t\, \mathrm{d}\mathcal{H}^n(x) \notag \\
	& \leq r\cdot \bigg( 1 + \frac{r}{\gamma_1(\phi) \gamma_2(G)}\bigg)^{n-1}\cdot\int_{B^\phi_r(E) \cap \partial G} \phi(\nu(G)(x))\, \mathrm{d}\mathcal{H}^n(x)
\end{flalign}
for every $ r > 0 $.

\subsection*{Bulk materials with soft inclusions}

We state here a version of a relaxation result for bulk materials from \cite{BraidesChambolleSolci:07,SantilliSchmidt:21} adapted to our needs. Suppose $W : \R^{3 \times 3} \to \R$ is Borel function that satisfies the growth condition $c|X|^2 - C \le W(X) \le C|x|^2+ C$ for suitable $c, C > 0$ and all $X \in \R^{3 \times 3}$. We consider the functionals $\mathcal{G} : W^{1,2}(\Omega;\R^3) \times \mathcal{F}(\Omega) \to \R$, $\mathcal{G}^{\rm rel} : SBV^2(\Omega; \R^3) \times \mathcal{F}(\Omega) \to \R$ given by 
\begin{align*} 
 \mathcal{G}(y, D) 
	&=  \int_{\Omega \setminus D} W(\nabla y) \, \mathrm{d}x 
		+  \int_{\Omega \cap \partial^\ast D} \psi(\nu(D))\, \mathrm{d}\mathcal{H}^{n-1}, \\
 \mathcal{G}^{\rm rel}(y, D) 
	&=  \int_{\Omega \setminus D} W^{\rm qc} (\nabla y) \, \mathrm{d}x \\
		& \qquad + 2  \int_{J_y \cap D^0} \psi(\nu(y))\, \mathrm{d}\mathcal{H}^{n-1} +  \int_{\Omega \cap \partial^\ast D} \psi(\nu(D))\, \mathrm{d}\mathcal{H}^{n-1}.
\end{align*}

\begin{theorem}\label{theo:bulk-relax}
\begin{itemize} 
\item[(i)] Whenever $(y_k) \subset W^{1,2}(\Omega; \R^{3})$ and $(D_k) \subset \mathcal{F}(\Omega)$ are such that $y_k \to y$ in $L^1(\Omega; \R^{3})$, $\limsup_{k} \|y_k\|_{L^\infty} < \infty$ for some $y \in SBV^2_\infty(\Omega; \R^{3})$ and $\chi_{D_k} \to \chi_{D}$ in $L^1(\Omega)$ and for some $D \in \mathcal{F}(\Omega)$, then one has 
\[ \liminf_{k \to \infty} \mathcal{G}(y_k, D_k) 
   \ge \mathcal{G}^{\rm rel}(y, D), \] 
 \item[(ii)] For each $(y, D) \in  SBV^2_\infty(\Omega; \R^{3}) \times \mathcal{F}(\Omega)$ and $c_1, c_2, \ldots \in (0,\mathcal{L}^n(\Omega)]$ with $c_k \to \mathcal{L}^n(D)$ there are $(y_k) \subset C^{\infty}(\overline{\Omega}; \R^3)$ with $y_k \to y$ in $L^1(\Omega; \R^{3})$ and $D_k \subset \Omega$ with smooth boundary such that $\chi_{D_k} \to \chi_{D}$ in $L^1(\Omega)$, $\mathcal{L}^n(D_k) = c_k$ for all $k$ 
and
\[ \lim_{k \to \infty} \mathcal{G}(y_k, D_k) 
   = \mathcal{G}^{\rm rel}(y, D). \] 
and in addition $\limsup_k \| y_k \|_{L^\infty} = \| y \|_{L^\infty}$. 
\end{itemize}
\end{theorem}

We refer to \cite[Theorem~3.1 and Remark~3.3]{SantilliSchmidt:21}

\section{Minimal droplet assumption and curvature}

In order to state and prove the main theorem of this section (Theorem \ref{thm: minimal droplet assumption and curvature}) we need some preparation.

\paragraph{Varifolds} We recall a few basic notions in varifold's theory. We refer to \cite{Allard72} for details. Let $ U \subset \mathbb{R}^n $ be an open set and  let $ 1 \leq k \leq n $ be an integer. Denoting with $ G(n,k) $ is the Grassmannian  of all $ k $-dimensional linear subspaces of $ \mathbb{R}^k $, a $k$-dimensional varifold on $ U $ is a Radon measure $ V $ on $ U \times G(n,k) $. If $ \pi : U \times G(n,k) \rightarrow U $ is the projection onto the first component, then we define the \emph{mass of $V$} to be the Radon measure over $ U $ given by
\begin{equation*}
	\| V \|(S) = V(S \times G(n,k)) \qquad \textrm{for $ S \subset U $.}
\end{equation*}
By the classical disintegration theorem we see that there exists a $ \| V \| $-measurable map $ V^{(\cdot)} $ with values in the space of probability measures on $ G(n,k) $ such that 
\begin{equation*}
	\int \varphi(x,S)\, dV(x,S) = \int \bigg(\int \varphi (x,S)\, dV^{(x)}(S)\bigg)\, d\|V\|(x)
\end{equation*}
for every $ \varphi \in C_c(U \times G(n,k)) $. We say that $ V $ is $ k $-dimensional \emph{integral} varifold if and only if there exists a countably $ \mathcal{H}^k $-rectifiable set $ M \subset U $ and an positive integer-valued map $ \theta $ on $ M $ such that 
\begin{equation*}
\| V \| = \theta \,\mathcal{H}^k \restrict M 
\end{equation*}
and $ V^{(x)} $ is the Dirac-delta concentrated on the singleton $ \{  \Tan^k(\| V \|,x)  \} $ for $ \| V \| $ a.e.\ $ x \in U $. We denote with $ IV_k(U) $ the space of all integral $ k $-dimensional varifolds on $ U $.

If $ M $ is an  $ \mathcal{H}^k $-rectifiable subset of $ U $ then $ v(M) $ is the $ k $-dimensional varifold characterized by
\begin{equation*}
\int \varphi(x,S)\, dv(M)(x,S) = \int_M \varphi(x, \Tan^k(\mathcal{H}^k \restrict M,x))\, d\mathcal{H}^k(x)
\end{equation*}
for every $ \varphi \in C^1_c(U) $. Notice that $ \| v(M)\| = \mathcal{H}^k \restrict M $.

A fundamental notion in varifold's theory is that of first variation. Suppose $ V $ is a $ k $-dimensional varifold on $ U $. The \emph{Euclidean first variation} of $ V $ is the linear map $ \delta V : C^1_c(U, \mathbb{R}^n) \rightarrow \mathbb{R} $ defined by 
\begin{equation*}
\delta V(g) = \int \textrm{div}_S g(x)\, dV(x,S);
\end{equation*}
see \cite[4.1, 4.2]{Allard72} for further details. If the total variation measure $ \| \delta V \| $ of $ \delta V $ (see \cite[4.2]{Allard72}) is a Radon measure on $ U $ absolutely continuous with respect to $ \| V \| $, then 
\begin{equation*}
\delta V(g)  = -\int \langle H_V(x), g\rangle \, d\| V \|(x) \qquad \textrm{for $ g \in C^1_c(U, \mathbb{R}^n) $,}
\end{equation*}
where $ H_V \in L^1(\| V \|, \mathbb{R}^n) $ is called \emph{generalized (Euclidean) mean curvature vector of $ V $.} In this case, employing the $ C^2 $-rectifiability of (integral) varifolds proved in  \cite{Menne13} (see also \cite{Santilli21}), one can readily see that there exists a countable collection of $ k $-dimensional $ C^2 $-submanifolds $ M_i \subset U $ such that $ \| V \|(U \setminus \bigcup_{i=1}^\infty M_i) =0 $ and there exists for $ \| V \| $ a.e.\ $ x \in U $ a symmetric bi-linear form 
\begin{equation*}
\II_V(x) : \Tan^k(\| V \|, x) \times \Tan^k(\| V \|,x) \rightarrow \Nor^k(\| V \|, x)
\end{equation*}
such that $ \textrm{trace } \II_V(x) = H_V(x) $. Here $ \Tan^k(\| V \|, x ) $ is the approximate tangent space of the $ k $-dimensional rectifiable measure $ \| V \| $ and $ \Nor^k(\| V \|, x) = \{v : \langle v, u \rangle =0 \quad \textrm{for $ u \in \Tan^k(\| V \|,x) $}\} $. Indeed $ \II_V $ satisfies the \emph{locality property}:  if $ M \subset U $ is a $ k $-dimensional submanifold of class $ C^2 $ then $ \II_V(x) $ coincides with the second fundamental form of $ M $ at $ \| V \| $ a.e.\ $ x \in M $. We refer to $\II_V $ as the \emph{generalized second fundamental form} of $ V $. We remark that if $ V $ is a curvature varifold (see \cite{Hutchinson86}), then $ \II_V $ coincides with its
variationally defined second fundamental form, see \cite[Remark 4.10]{Menne13}.

\paragraph{}The following lemma follows from classical arguments. 
\begin{lemma}\label{structure of sets of f.p.}
	Let $ U \subset \mathbb{R}^n $ be open and $ E \in \mathcal{F}(U) $ such that $ \mathcal{H}^{n-1}(U \cap \overline{\partial^\ast E} \setminus \partial^\ast E) =0 $.
	
	Then there exists an open set $ P \subset U  $ such that
	\begin{equation*}
	\mathcal{H}^{n-1}(U \cap \partial P \setminus \partial^\ast P) =0, \qquad \mathcal{L}^n((P \setminus E) \cup (E \setminus P)) =0
	\end{equation*}
 \begin{equation*}
\textrm{and}\qquad \spt(\mathcal{H}^{n-1}\restrict \partial^\ast P) = \partial P.
\end{equation*}
\end{lemma}

\begin{proof}
	We define $	P =   \{ x\in U : \mathcal{L}^n (B(x, \rho) 	\setminus E) =0  \; \textrm{for some $ \rho > 0 $} \} $ and $Q =\{x\in U: \mathcal{L}^n (B(x, \rho) 	\cap E) =0  \; \textrm{for some $ \rho > 0 $} \} $ and we notice that they are open subsets of $ \mathbb{R}^n $. It follows from the relative isoperimetric inequality, see \cite[eq.\ (3.43), pag.\ 152]{AmbrosioFuscoPallara:00} that 
	\begin{equation}\label{structure of sets of f.p.eq1}
		\textrm{spt} \big(\mathcal{H}^{n-1} \restrict \partial^\ast E\big) =\mathbb{R}^n  \setminus (P \cup Q).
	\end{equation}
As $ E \setminus P =  (E \cap Q) \cup (E \setminus (P \cup Q))  $ and $ \mathcal{L}^n\big(U \cap \textrm{spt} \big(\mathcal{H}^{n-1} \restrict \partial^\ast E\big)\big) =0 $, we infer from standard density results that
	\begin{equation*}
		\mathcal{L}^n(E \setminus P) = 0 \quad \textrm{and} \quad \mathcal{L}^n(P \setminus E) =0.
	\end{equation*}
	We deduce that $ \partial^\ast P = \partial^\ast E $ and, since $ \partial P \subset \spt \mathcal{H}^{n-1} \restrict \partial^\ast E \subset \overline{\partial^\ast E} $ by \eqref{structure of sets of f.p.eq1}, we conclude from the hypothesis that $\mathcal{H}^{n-1}(U \cap \partial P \setminus \partial^\ast P) =0$.
\end{proof}

For the next lemma we first need to recall a few notions about the theory of curvature for arbitrary closed sets. We refer to \cite{HugSantilli} for details. If $ A \subset \mathbb{R}^n $ we define
\begin{equation*}
	N(A) = \{  (a, \eta)  \in \overline{A}\times \mathbb{S}^{n-1} : \dist_{|\cdot|}(a + s\eta, A) = s \; \textrm{for some $ s > 0 $} \}
\end{equation*}
and $ N(A,a) = \{ \eta \in \mathbb{S}^n : (a, \eta) \in N(A)   \} $.  The set $ \partial^v_+ A $ is defined as the set of all $ a \in \overline{A} $ for which there exist $ s > 0 $ and $ \eta \in N(A,a) $ such that $ \dist_{|\cdot|}(a-s\eta, \mathbb{R}^n \setminus A) = s $ lies in the interior of $ A $ and $ \dist_{|\cdot|}(a + s\eta, A) = s $. Notice that $ N(A,a) $ contains only one vector for each $ a \in \partial^v_+ A $ and $ \partial^v_+ A \subset \partial^\ast A $. The set $ N(A) $ is a countably $ \mathcal{H}^{n-1} $-rectifiable subsets of $ \overline{A} \times \mathbb{S}^{n-1} $ and $ \kappa_{A,1} \leq \ldots \leq \kappa_{A,n-1} $ denotes the principal curvatures of $ A $, see \cite[Defnition 3.6, Remark 3.7]{HugSantilli}. 
\begin{lemma}\label{normal bundle of sets of finite perimeters}
	Suppose $ U \subset \mathbb{R}^{n} $ is open, $ E \in \mathcal{F}(U) $, $ V = v(U \cap \partial^\ast E) \in IV_{n-1}(U) $, $ 0 \leq \kappa < \infty $ and 
	$ \| \delta V \| \leq \kappa \| V \|. $
	 
	 Then there exists $ C \subset U $ relatively closed in $ U $ with non empty interior  such that
	\begin{enumerate}[label=(\alph*)]
		\item\label{normal bundle of sets of finite perimeters: 1} $ \mathcal{L}^n((C \setminus E) \cup (E \setminus C)) =0 $,
		\item\label{normal bundle of sets of finite perimeters: 3} $ \mathcal{H}^{n-1}(U \cap \partial C \setminus \partial^v_+ C) =0 $,
		\item\label{normal bundle of sets of finite perimeters: 4} $ v(U \cap \partial^\ast E) = v(U \cap \partial C) $,
		\item\label{normal bundle of sets of finite perimeters: 5} $ N(C,a) = \{ \nu(C)(a)  \} $ for every $ a \in U \cap \partial^{v}_+C $ and $$\mathcal{H}^{n-1}\big((N(C)\cap (U \times \mathbb{S}^{n-1})) \setminus \{(a, \nu(C)(a)): a \in U \cap \partial^v_+ C\}\big) =0, $$
	\item\label{normal bundle of sets of finite perimeters: 6} The principal curvatures $ \kappa_{C,1}(a, \nu(C)(a)) \leq \ldots \leq \kappa_{C,n}(a, \nu(C)(a)) $ are the eigenvalues of $ \langle  -\II_V(a)(\cdot, \cdot),  \nu(C)(a) \rangle $ for $ \| V \| $ a.e.\ $ a \in U $.
	\end{enumerate}
\end{lemma}

\begin{proof}
We notice that the $ (n-1) $-dimensional density $ \theta $ of $ \| V \| = \mathcal{H}^{n-1}\restrict \partial^\ast E $ is $ \| V \| $-almost everywhere equal to $ 1 $ and it is an upper semi-continuous function on $ U $ by \cite[8.6]{Allard72}. It follows that
	\begin{equation*}
		\theta(x) = 1 \qquad \textrm{for every $ x \in U \cap \overline{\partial^\ast E} $,}
	\end{equation*}
	whence we infer from standard density results that $ \mathcal{H}^{n-1}(U \cap \overline{\partial^\ast E} \setminus \partial^\ast E) =0 $. We employ now \ref{structure of sets of f.p.} to find an open subset $ P \subset U $ such that, for $ C = \overline{P}\cap U$, it holds that
	\begin{equation}\label{normal bundle of sets of finite perimeters eq1}
	\mathcal{H}^{n-1}	(U \cap \partial C \setminus \partial^\ast C)=0,
	\end{equation}
	\begin{equation}\label{normal bundle of sets of finite perimeters eq2}
	 \spt(\mathcal{H}^{n-1}\restrict \partial^\ast C) = \partial C \qquad  \textrm{and} \qquad \mathcal{L}^n((E \setminus C) \cup (C \setminus E)) =0. 
	\end{equation}	
Then \ref{normal bundle of sets of finite perimeters: 1} is proved and $ \partial^\ast E = \partial^\ast C $.	Since $ \rho^{-1}(C - x) $ converges in measure to an halfspace as $ \rho \to 0+ $ for $ \mathcal{H}^{n-1} $ a.e.\ $ x \in U \cap \partial^\ast C $, denoting with $ \bm{p}: \mathbb{R}^n \times \mathbb{S}^{n-1} \rightarrow \mathbb{R}^n $ the projection onto the first coordinate, we infer from \eqref{normal bundle of sets of finite perimeters eq1} using \cite[Lemma 3.25(c)]{HugSantilli} that 
	\begin{equation*}
	\mathcal{H}^{n-1}( U \cap \bm{p}(N(C)) \setminus \partial^v_+ C) =0.
	\end{equation*}
Notice that \ref{normal bundle of sets of finite perimeters: 4} also follows from \eqref{normal bundle of sets of finite perimeters eq1}. Moreover, since $ \mathcal{H}^{n-1}(U \cap \partial C \setminus \bm{p}(N(C))) =0 $ (this follows from \eqref{normal bundle of sets of finite perimeters eq2} and  \cite[Theorem 1.3]{Santilli21}), we readily obtain the conclusion in \ref{normal bundle of sets of finite perimeters: 3}. The first part of \ref{normal bundle of sets of finite perimeters: 5} is clear from the definition of $ \partial^v_+ C $; see \cite[Remark 2.11]{HugSantilli}. Since, by \cite[Theorem 3.8]{Santilli20}, it holds that 
\begin{equation}\label{normal bundle of sets of finite perimeters eq3}
\mathcal{H}^{n-1}(N(C) \cap (S \times \mathbb{S}^{n-1})) =0 \quad \textrm{whenever $ S \subset U $ with $ \mathcal{H}^{n-1}(S) =0 $,}
\end{equation}
we infer the second part of \ref{normal bundle of sets of finite perimeters: 5} from \ref{normal bundle of sets of finite perimeters: 3}.
	
Finally the assertion in \ref{normal bundle of sets of finite perimeters: 6} follows combining \cite[Lemma 6.1]{Santilli20a} (see also \cite[Definition 4.7 and 4.9]{Santilli20a}), the locality property of $ \II_V $ and \eqref{normal bundle of sets of finite perimeters eq3}.
\end{proof}

\paragraph{} We use the notation $ \Omega_s = \omega \times  (0,s) $, where $ s > 0 $ and $ \omega $ is an open set in $ \mathbb{R}^2 $.
	\begin{theorem}\label{thm: minimal droplet assumption and curvature}
		Let $(E_h)_{h \in (0,1)} \subset \mathcal{F}(\Omega_{h}) $, $ V_h = v(\Omega_{h} \cap \partial^\ast E_h) \in IV_2(\Omega_{h}) $. Suppose there exist $ 0 \leq C < \infty $ and $ 0 < \epsilon_0 < 1 $ such that
			\begin{equation*}
			\mathcal{L}^3((E_h^{(t)} \setminus E_h) \cap (\Omega_h^- \setminus \Omega_{h-s}^-)) \leq Cst \quad \textrm{for every $ h \in (0,1) $ and	$ t,s \in (0, \epsilon_0 h) $,}
			\end{equation*}
		\begin{equation*}
			\textrm{$\| \delta V_h \| $ is a Radon measure on $ \Omega_{h} $ absolutely continuous with respect to $ \| V_h \| $}
		\end{equation*}
		\begin{equation*}
		\textrm{and} \quad 	|\II_{V_h}(x)| \leq Ch^{-1} \qquad \textrm{for $  \| V_h \| $ a.e.\ $ x \in \Omega_{h} $.}
		\end{equation*}
		
		If $ \psi $ is the Euclidean norm on $ \mathbb{R}^3 $ then $(E_h)_{h \in (0,1)}$ satisfies the $ \psi $-minimal droplet assumption in $ \Omega_h $.
	\end{theorem}
	
	\begin{proof}
		We choose a sequence $ (\epsilon_i)_{i \geq 1} $ such that $ 0 < \epsilon_{i} < \epsilon_{i-1} $ for every $ i \geq 1 $ and $ \epsilon_i \searrow 0 $ as $ i \to \infty $. We define $ \zeta : (0, \epsilon_0) \rightarrow (0, +\infty) $ by $ \zeta(s) = (C^2  + 2C)\epsilon_i $ for $ \epsilon_{i+1} < s \leq \epsilon_i $ and for $ i \geq 0 $. By Lemma \ref{normal bundle of sets of finite perimeters} we can assume that $ E_h $ is relatively closed in $ \Omega_{h} $ and \ref{normal bundle of sets of finite perimeters: 3}, \ref{normal bundle of sets of finite perimeters: 5} and \ref{normal bundle of sets of finite perimeters: 6} of Lemma \ref{normal bundle of sets of finite perimeters} hold with $ C $ and $ U $ replaced by $ E_h $ and $ \Omega_{h} $. In particular, 
		\begin{equation*}
			|\kappa_{A_h, i}(a, \nu(E_h)(a))| \leq \frac{C}{h} \quad \textrm{for $ \mathcal{H}^2 $ a.e.\ $ a \in \Omega_{h} \cap \partial E_h $ and $ i =1,2 $.}
		\end{equation*}
	For $ \epsilon_{i+1} < s \leq \epsilon_i $ we can use the disintegration formula in \cite[Theorem 3.18]{HugSantilli} to obtain
		\begin{equation*}
		\mathcal{L}^3\big((E_h^{(sh)} \setminus E_h) \cap \Omega_{h - \epsilon_i h}^-   \big)  \leq sh (1 + Cs)^2  \mathcal{H}^2(\Omega_{h} \cap \partial^v_+ E_h) 
	\end{equation*}
and, since by assumption $ \mathcal{L}^3((E_h^{(sh)} \setminus E_h) \cap (\Omega_h^- \setminus \Omega_{h - \epsilon_i h}^-  )) \leq C \epsilon_i sh^2 $, we infer that 
\begin{flalign*}
\mathcal{L}^3\big((E_h^{(sh)} \setminus E_h) \cap \Omega_{h}^-   \big) & \leq sh (1 + Cs)^2  \mathcal{H}^2(\Omega_{h} \cap \partial^v_+ E_h)  + C\epsilon_i sh^2 \\
& \leq sh (1 + \zeta(s))\mathcal{H}^2(\Omega_h \cap \partial^v_+ E_h) + sh^2 \zeta(s).
\end{flalign*}
\end{proof}

\end{appendix}

\bibliographystyle{abbrv}
\bibliography{BZK}

\begin{thebibliography}{10}

\bibitem{AbelsMoraMueller:11}
H.~Abels, M.~G. Mora, and S.~M\"{u}ller.
\newblock The time-dependent von {K}\'{a}rm\'{a}n plate equation as a limit of
  3d nonlinear elasticity.
\newblock {\em Calc. Var. Partial Differential Equations}, 41(1-2):241--259,
  2011.

\bibitem{AcerbiButtazzoPercivale:88}
E.~Acerbi, G.~Buttazzo, and D.~Percivale.
\newblock Thin inclusions in linear elasticity: A variational approach.
\newblock {\em Journal f\"ur die reine und angewandte Mathematik}, 386:99--115,
  1988.

\bibitem{AcerbiButtazzoPercivale:91}
E.~Acerbi, G.~Buttazzo, and D.~Percivale.
\newblock A variational definition of the strain energy for an elastic string.
\newblock {\em Journal of Elasticity}, 25(2):137--148, 1991.

\bibitem{Allard72}
W.~K. Allard.
\newblock On the first variation of a varifold.
\newblock {\em Ann. of Math. (2)}, 95:417--491, 1972.

\bibitem{AmbrosioFuscoPallara:00}
L.~Ambrosio, N.~Fusco, and D.~Pallara.
\newblock {\em Functions of bounded variation and free discontinuity problems}.
\newblock Oxford Mathematical Monographs. The Clarendon Press, Oxford
  University Press, New York, 2000.

\bibitem{AnzellottiBaldoPercivale:94}
G.~Anzellotti, S.~Baldo, and D.~Percivale.
\newblock Dimension reduction in variational problems, asymptotic development
  in {$\Gamma$}-convergence and thin structures in elasticity.
\newblock {\em Asymptotic Analysis}, 9(1):61--100, 1994.

\bibitem{Babadjian:06}
J.-F. Babadjian.
\newblock Quasistatic evolution of a brittle thin film.
\newblock {\em Calc. Var. Partial Differential Equations}, 26(1):69--118, 2006.

\bibitem{BartelsBonitoHornung:21}
S.~Bartels, A.~Bonito, and P.~Hornung.
\newblock Modeling and simulation of thin sheet folding.
\newblock Online preprint arXiv:2108.00937 [math.NA], 2021.

\bibitem{BhattacharyaLewickaSchaeffner:16}
K.~Bhattacharya, M.~Lewicka, and M.~Sch{\"a}ffner.
\newblock Plates with incompatible prestrain.
\newblock {\em Archive for Rational Mechanics and Analysis}, 221:143--181,
  20016.

\bibitem{BlakeZisserman:87}
A.~Blake and A.~Zisserman.
\newblock {\em Visual reconstruction}.
\newblock MIT Press Series in Artificial Intelligence. MIT Press, Cambridge,
  MA, 1987.

\bibitem{BoehnleinNeukammPadilla-GarzaSander:22}
K.~B{\"o}hnlein, S.~Neukamm, D.~Padilla-Garza, and O.~Sander.
\newblock A homogenized bending theory for prestrained plates.
\newblock Online preprint arXiv:2203.11098 [math.AP], 2022.

\bibitem{BonnetierChambolle:02}
E.~Bonnetier and A.~Chambolle.
\newblock Computing the equilibrium configuration of epitaxially strained
  crystalline films.
\newblock {\em SIAM J. Appl. Math.}, 62(4):1093--1121, 2002.

\bibitem{BraidesChambolleSolci:07}
A.~Braides, A.~Chambolle, and M.~Solci.
\newblock A relaxation result for energies defined on pairs set-function and
  applications.
\newblock {\em ESAIM Control Optim. Calc. Var.}, 13(4):717--734, 2007.

\bibitem{BraidesFonseca:01}
A.~Braides and I.~Fonseca.
\newblock Brittle thin films.
\newblock {\em Appl. Math. Optim.}, 44(3):299--323, 2001.

\bibitem{BraunSchmidt:19}
J.~Braun and B.~Schmidt.
\newblock An atomistic derivation of von-{K}ármán plate theory.
\newblock {\em Networks and Heterogeneous Media}, Online first:--, 2022.

\bibitem{CarrieroMicheleLeaciTomarelli:96}
M.~Carriero, A.~Leaci, and F.~Tomarelli.
\newblock A second order model in image segmentation: {B}lake \& {Z}isserman
  functional.
\newblock In {\em Variational methods for discontinuous structures ({C}omo,
  1994)}, volume~25 of {\em Progr. Nonlinear Differential Equations Appl.},
  pages 57--72. Birkh\"{a}user, Basel, 1996.

\bibitem{CarrieroMicheleLeaciTomarelli:97}
M.~Carriero, A.~Leaci, and F.~Tomarelli.
\newblock Strong minimizers of {B}lake \& {Z}isserman functional.
\newblock volume~25, pages 257--285 (1998). 1997.
\newblock Dedicated to Ennio De Giorgi.

\bibitem{CarrieroMicheleLeaciTomarelli:15}
M.~Carriero, A.~Leaci, and F.~Tomarelli.
\newblock A survey on the {B}lake-{Z}isserman functional.
\newblock {\em Milan J. Math.}, 83(2):397--420, 2015.

\bibitem{ChambolleGiacominiPonsiglione:07}
A.~Chambolle, A.~Giacomini, and M.~Ponsiglione.
\newblock Piecewise rigidity.
\newblock {\em J. Funct. Anal.}, 244(1):134--153, 2007.

\bibitem{ChambolleSolci:07}
A.~Chambolle and M.~Solci.
\newblock Interaction of a bulk and a surface energy with a geometrical
  constraint.
\newblock {\em SIAM J. Math. Anal.}, 39(1):77--102, 2007.

\bibitem{Ciarlet-II:97}
P.~G. Ciarlet.
\newblock {\em Mathematical elasticity. {V}ol. {II}: Theory of plates},
  volume~27 of {\em Studies in Mathematics and its Applications}.
\newblock North-Holland Publishing Co., Amsterdam, 1997.

\bibitem{Ciarlet-III:00}
P.~G. Ciarlet.
\newblock {\em Mathematical elasticity. {V}ol. {III}: Theory of shells},
  volume~29 of {\em Studies in Mathematics and its Applications}.
\newblock North-Holland Publishing Co., Amsterdam, 2000.

\bibitem{ContiDolzmann:09}
S.~Conti and G.~Dolzmann.
\newblock {$\Gamma$}-convergence for incompressible elastic plates.
\newblock {\em Calc. Var. Partial Differential Equations}, 34(4):531--551,
  2009.

\bibitem{CrismaleFriedrich:20}
V.~Crismale and M.~Friedrich.
\newblock Equilibrium configurations for epitaxially strained films and
  material voids in three-dimensional linear elasticity.
\newblock {\em Arch. Ration. Mech. Anal.}, 237(2):1041--1098, 2020.

\bibitem{deBenitoSchmidt:19b}
M.~de~Benito~Delgado and B.~Schmidt.
\newblock Energy minimizing configurations of pre-strained multilayers.
\newblock {\em Journal of Elasticity}, 140:303--335, 2020.

\bibitem{deBenitoSchmidt:19a}
M.~de~Benito~Delgado and B.~Schmidt.
\newblock A hierarchy of multilayered plate models.
\newblock {\em ESAIM Control Optim. Calc. Var.}, 27(suppl.):Paper No. S16, 35,
  2021.

\bibitem{DeRosaKolasinskiSantilli}
A.~De~Rosa, S.~Kolasi\'{n}ski, and M.~Santilli.
\newblock Uniqueness of critical points of the anisotropic isoperimetric
  problem for finite perimeter sets.
\newblock {\em Arch. Ration. Mech. Anal.}, 238(3):1157--1198, 2020.

\bibitem{Euler}
L.~Euler.
\newblock {M}ethodus {I}nveniendi {L}ineas {C}urvas, {A}dditamentum {I}: {D}e
  {C}urvis {E}lasticis (1744).
\newblock In {\em Opera Omnia Ser.\ Prima}, volume XXIV, pages 231--297. {Orell
  F{\"u}ssli, Bern}, 1952.

\bibitem{FriedrichKreutzZemas:21}
M.~Friedrich, L.~Kreutz, and K.~Zemas.
\newblock Geometric rigidity in variable domains and derivation of linearized
  models for elastic materials with free surfaces.
\newblock Online preprint arXiv:2107.10808 [math.AP], 2021.

\bibitem{FriedrichSchmidt:15}
M.~Friedrich and B.~Schmidt.
\newblock A quantitative geometric rigidity result in sbd.
\newblock Online preprint arXiv:1503.06821 [math.AP], 2015.

\bibitem{FrieseckeJamesMoraMueller:03}
G.~Friesecke, R.~D. James, M.~G. Mora, and S.~M\"{u}ller.
\newblock Derivation of nonlinear bending theory for shells from
  three-dimensional nonlinear elasticity by {G}amma-convergence.
\newblock {\em C. R. Math. Acad. Sci. Paris}, 336(8):697--702, 2003.

\bibitem{FrieseckeJamesMueller:02}
G.~Friesecke, R.~D. James, and S.~M\"uller.
\newblock A theorem on geometric rigidity and the derivation of nonlinear plate
  theory from three-dimensional elasticity.
\newblock {\em Communications on Pure and Applied Mathematics},
  55(11):1461--1506, 2002.

\bibitem{FrieseckeJamesMueller:06}
G.~Friesecke, R.~D. James, and S.~M\"uller.
\newblock A hierarchy of plate models derived from nonlinear elasticity by
  {$\Gamma$}-convergence.
\newblock {\em Archive for Rational Mechanics and Analysis}, 180(2):183--236,
  2006.

\bibitem{HornungNeukammVelvic14}
P.~Hornung, S.~Neukamm, and I.~Vel\v{c}i\'{c}.
\newblock Derivation of a homogenized nonlinear plate theory from 3d
  elasticity.
\newblock {\em Calc. Var. Partial Differential Equations}, 51(3-4):677--699,
  2014.

\bibitem{HornungPawelczykVelvic18}
P.~Hornung, M.~Pawelczyk, and I.~Vel\v{c}i\'{c}.
\newblock Stochastic homogenization of the bending plate model.
\newblock {\em J. Math. Anal. Appl.}, 458(2):1236--1273, 2018.

\bibitem{HugSantilli}
D.~Hug and M.~Santilli.
\newblock Curvature measures and soap bubbles beyond convexity.
\newblock Online preprint arXiv:2204.04909 [math.MG], 2022.

\bibitem{Hutchinson86}
J.~E. Hutchinson.
\newblock Second fundamental form for varifolds and the existence of surfaces
  minimising curvature.
\newblock {\em Indiana Univ. Math. J.}, 35(1):45--71, 1986.

\bibitem{Kirchhoff}
G.~Kirchhoff.
\newblock \"{U}ber das {G}leichgewicht und die {B}ewegung einer elastischen
  {S}cheibe.
\newblock {\em J. Reine Angew. Math.}, 40:51--88, 1850.

\bibitem{LeDretRaoult:95}
H.~Le~Dret and A.~Raoult.
\newblock The nonlinear membrane model as variational limit of nonlinear
  three-dimensional elasticity.
\newblock {\em Journal de Math\'ematiques Pures et Appliqu\'ees},
  74(6):549--578, 1995.

\bibitem{LewickaLucic:18}
M.~Lewicka and D.~Lu\v{c}i\'{c}.
\newblock Dimension reduction for thin films with transversally varying
  prestrian: the oscillatory and the non-oscillatory case, 2018.
\newblock Preprint, available at https://arxiv.org/abs/1807.02060.

\bibitem{LewickaMahadevanPakzad:11}
M.~Lewicka, L.~Mahadevan, and M.~R. Pakzad.
\newblock The {{F\"oppl}}-von {{K\'arm\'an}} equations for plates with
  incompatible strains.
\newblock {\em Proceedings of the Royal Society of London. Series A.
  Mathematical, Physical and Engineering Sciences}, 467(2126):402--426, 2011.

\bibitem{LewickaMoraPakzad:08}
M.~Lewicka, M.~G. Mora, and M.~R. Pakzad.
\newblock Shell theories arising as low energy {$\Gamma$}-limit of 3d nonlinear
  elasticity.
\newblock {\em Annali della Scuola normale superiore di Pisa, Classe di
  scienze}, 9(2):253--295, 2008.

\bibitem{Pakzad}
Z.~Liu and M.~R. Pakzad.
\newblock Rigidity and regularity of codimension-one {S}obolev isometric
  immersions.
\newblock {\em Ann. Sc. Norm. Super. Pisa Cl. Sci. (5)}, 14(3):767--817, 2015.

\bibitem{Love}
A.~E.~H. Love.
\newblock {\em A treatise on the {M}athematical {T}heory of {E}lasticity}.
\newblock Dover Publications, New York, 1944.
\newblock Fourth Ed.

\bibitem{LussardiVilla:16}
L.~Lussardi and E.~Villa.
\newblock A general formula for the anisotropic outer {M}inkowski content of a
  set.
\newblock {\em Proc. Roy. Soc. Edinburgh Sect. A}, 146(2):393--413, 2016.

\bibitem{MaorShachar:19}
C.~Maor and A.~Shachar.
\newblock On the role of curvature in the elastic energy of non-{E}uclidean
  thin bodies.
\newblock {\em J. Elasticity}, 134(2):149--173, 2019.

\bibitem{Menne13}
U.~Menne.
\newblock Second order rectifiability of integral varifolds of locally bounded
  first variation.
\newblock {\em J. Geom. Anal.}, 23(2):709--763, 2013.

\bibitem{MoraMuellerSchultz:07}
M.~G. Mora, S.~M\"{u}ller, and M.~G. Schultz.
\newblock Convergence of equilibria of planar thin elastic beams.
\newblock {\em Indiana Univ. Math. J.}, 56(5):2413--2438, 2007.

\bibitem{MuellerPakzad:08}
S.~M\"{u}ller and M.~R. Pakzad.
\newblock Convergence of equilibria of thin elastic plates---the von
  {K}\'{a}rm\'{a}n case.
\newblock {\em Comm. Partial Differential Equations}, 33(4-6):1018--1032, 2008.

\bibitem{NeukammVelvic13}
S.~Neukamm and I.~Vel\v{c}i\'{c}.
\newblock Derivation of a homogenized von-{K}\'{a}rm\'{a}n plate theory from
  3{D} nonlinear elasticity.
\newblock {\em Math. Models Methods Appl. Sci.}, 23(14):2701--2748, 2013.

\bibitem{Santilli20a}
M.~Santilli.
\newblock Fine properties of the curvature of arbitrary closed sets.
\newblock {\em Ann. Mat. Pura Appl. (4)}, 199(4):1431--1456, 2020.

\bibitem{Santilli20}
M.~Santilli.
\newblock Normal bundle and {A}lmgren's geometric inequality for singular
  varieties of bounded mean curvature.
\newblock {\em Bull. Math. Sci.}, 10(1):2050008, 24, 2020.

\bibitem{Santilli21}
M.~Santilli.
\newblock Second order rectifiability of varifolds of bounded mean curvature.
\newblock {\em Calc. Var. Partial Differential Equations}, 60(2):Paper No. 81,
  17, 2021.

\bibitem{SantilliSchmidt:21}
M.~Santilli and B.~Schmidt.
\newblock Two phase models for elastic membranes with soft inclusions.
\newblock Online preprint arXiv:2106.01120 [math.AP], 2022.

\bibitem{Schmidt:06}
B.~Schmidt.
\newblock A derivation of continuum nonlinear plate theory from atomistic
  models.
\newblock {\em Multiscale Model. Simul.}, 5(2):664--694, 2006.

\bibitem{Schmidt:07a}
B.~Schmidt.
\newblock Minimal energy configurations of strained multi-layers.
\newblock {\em Calculus of Variations and Partial Differential Equations},
  30(4):477--497, 2007.

\bibitem{Schmidt:07b}
B.~Schmidt.
\newblock Plate theory for stressed heterogeneous multilayers of finite bending
  energy.
\newblock {\em Journal de Math\'ematiques Pures et Appliqu\'ees},
  88(1):107--122, 2007.

\bibitem{Schmidt:17}
B.~Schmidt.
\newblock A {G}riffith-{E}uler-{B}ernoulli theory for thin brittle beams
  derived from nonlinear models in variational fracture mechanics.
\newblock {\em Math. Models Methods Appl. Sci.}, 27(9):1685--1726, 2017.

\bibitem{vonKarman}
T.~von {K}{\'a}rm{\'a}n.
\newblock {F}estigkeitsprobleme im {M}aschinenbau.
\newblock In {\em {E}ncyclop\"adie der {M}athematischen {W}issenschaften},
  volume IV/4, pages 311--385. Teubner, Leipzig, 1910.

\end{thebibliography}

\end{document}